\documentclass[reqno]{amsart}
\usepackage{graphicx}
\usepackage{amstext}
\usepackage{amssymb}
\usepackage{amsmath}
\usepackage{color}
\usepackage[utf8]{inputenc}
\newtheorem{theorem}{Theorem}[section]
\newtheorem{lemma}[theorem]{Lemma}

\newtheorem{proposition}[theorem]{Proposition}
\theoremstyle{definition}
\newtheorem{definition}[theorem]{Definition}

\theoremstyle{remark}
\newtheorem{remark}[theorem]{Remark}
\numberwithin{equation}{section}
\usepackage{hyperref}
\hypersetup{colorlinks,bookmarks=true,linkcolor=blue,citecolor=blue}

\begin{document}

\title{Blow-up dynamics for $L^2$-critical fractional Schr\"odinger equations}\def\rightmark{BLOW-UP FOR CRITICAL FRACTIONAL SCHR\"ODINGER EQUATIONS}

\author{Yang LAN}

\address{Department of Mathematics and Computer Science, University of Basel, Spiegelgasse 1, CH-4051 Basel, Switzerland}

\email{yang.lan@unibas.ch}

\keywords{Fractional Schr\"odinger equations, $L^2$-critical, blow-up dynamics}

\subjclass[2010]{Primary 35Q55; Secondary 35B35, 35B40, 35B44}

\begin{abstract}
In this paper, we consider the $L^2$-critical fractional Schr\"odinger equation $iu_t-|D|^{\beta}u+|u|^{2\beta}u=0$ with initial data $u_0\in H^{\beta/2}(\mathbb{R})$ and $\beta$ close to $2$. We show that if the initial data has negative energy and slightly supercritical mass, then the solution blows up in finite time. We also give a specific description for the blow-up dynamics. This is an extension of the works of F. Merle and P. Rapha\"el for $L^2$-critical Schr\"odinger equations. However, the nonlocal structure of this equation and the lack of some symmetries make the analysis more complicated, hence some new strategies are required.
\end{abstract}

\maketitle

\section{Introduction}
\subsection{Setting of the problem}
In this paper, we consider the following  fractional Schr\"odinger equation:
\begin{equation}
\label{CP}
\begin{cases}
iu_t-|D|^{\beta}u+|u|^{2\beta}u=0,\quad (t,x)\in\mathbb{R}\times\mathbb{R},\\
u|_{t=0}=u_0\in H^{\frac{\beta}{2}}(\mathbb{R}),
\end{cases}
\end{equation}
with $1\leq\beta<2$. Here $|D|^{\beta}$ is defined as following:
$$\widehat{|D|^{\beta}u}(\xi)=|\xi|^{\beta}\hat{u}(\xi).$$ 
This is an extension of the standard Schr\"odinger equation:
\begin{equation}
\begin{cases}
iu_t+\Delta u+|u|^{p-1}u=0,\quad (t,x)\in\mathbb{R}\times\mathbb{R}^d,\\
u|_{t=0}=u_0\in H^{1}(\mathbb{R}^d).
\end{cases}
\end{equation}

 The evolution problems with nonlocal dispersion like \eqref{CP} arise in various physical settings, which include continuum limits of lattice systems \cite{KLS}, models for wave turbulence \cite{CMMT,MMT2}, and gravitational collapse \cite{ES,FL,IP}. We also refer to \cite{CHHO,CHKL,KSM,KLR,W4} and the references therein for the background of the fractional Schr\"odinger model in mathematics, numerics and physics.

Let us review some basic properties of \eqref{CP}. The Cauchy problem \eqref{CP} is an infinite-dimensional Hamiltonian system, which has the following three conservation laws:
\begin{itemize}
\item Mass:
\begin{equation}
\label{11}
M(u(t))=\int |u(t)|^2=M(u_0).
\end{equation}
\item Energy:
\begin{equation}
\label{12}
E(u(t))=\frac{1}{2}\int \big||D|^{\frac{\beta}{2}}u\big|^2-\frac{1}{2(\beta+1)}\int|u(t)|^{2(\beta+1)}=E(u_0).
\end{equation}
\item Momentum:
\begin{equation}
P(u(t))=\Im\int u_x(t)\bar{u}(t)=P(u_0).
\end{equation}
\end{itemize}
The equation \eqref{CP} has the following symmetries:
\begin{itemize}
\item Phase: if $u(t,x)$ is a solution, then for all $\theta\in \mathbb{R}$, $u(t,x)e^{i\theta}$ is also a solution.
\item Translation: if $u(t,x)$ is a solution, then for all $t_0\in\mathbb{R}$, $x_0\in\mathbb{R}$, $u(t-t_0,x-x_0)$ is also a solution.
\item Scaling: if $u(t,x)$ is a solution, then for all $\lambda>0$,
\begin{equation}
\label{13}
u_{\lambda}(t,x)=\frac{1}{\lambda^{d/2}}u\bigg(\frac{t}{\lambda^{\beta}},\frac{x}{\lambda}\bigg)
\end{equation}
is also a solution.
\end{itemize}

\begin{remark}
Unlike the classic nonlinear Schr\"odinger equations, the equation \eqref{CP} does not have Galilean transform and psedo-conformal transform. This fact leads to some additional technical difficulty for obtaining blow-up results for \eqref{CP}.
\end{remark}

The Cauchy problem \eqref{CP} is $L^2$-critical since the $L^2$ norm is invariant under the scaling rule \eqref{13}: $$\|u_{\lambda}\|_{L^2}=\|u\|_{L^2},\;\;\textrm{for all } \lambda>0.$$
From \cite{GH,HS}, we know that the Cauchy problem \eqref{CP} is locally well-posed in the energy space $H^{\frac{\beta}{2}}$. More precisely, for all $u_0\in H^{\frac{\beta}{2}}$, there exists a unique solution $u(t)\in C([0,T),H^{\frac{\beta}{2}})$ to \eqref{CP}. Moreover, if the maximal lifetime $T$ is finite, then we have
\begin{equation}
\label{14}
\lim_{t\rightarrow T^-}\big\||D|^{\frac{\beta}{2}}u(t)\big\|_{L^2}=+\infty.
\end{equation}
However, unlike the $L^2$-critical Schr\"odinger equation which is locally well-posed in $L^2$ (see \cite{TW}), it is not known whether the Cauchy problem \eqref{CP} is well-posed in the critical space $L^2$ when $1\leq\beta<2$. Moreover in \cite{CP}, Choffrut and Pocovnicu proved that in the half-wave case ($\beta=1$), the Cauchy problem \eqref{CP} is ill-posed in $H^s$ if $s<\frac{1}{2}$.

There exists a special class of solutions to \eqref{CP} called the {\it solitary waves} generated by
\begin{equation}
\label{15}
u(t,x)=e^{it}Q_{\beta}(x),
\end{equation}
where $Q_{\beta}\in H^{\frac{\beta}{2}}$ is  a solution to the following equation:
\begin{equation}
\label{16}
|D|^{\beta}Q_{\beta}+Q_{\beta}-|Q_{\beta}|^{2\beta}Q_{\beta}=0.
\end{equation}
From \cite{FL, FLS}, we know that there exists a unique radial nonnegative $H^{\frac{\beta}{2}}$ solution to \eqref{16}, with
\begin{equation}
\label{17}
Q_{\beta}(y)\sim \frac{1}{|y|^{1+\beta}},\;\;\text{as }|y|\rightarrow\infty.
\end{equation}
We call $Q_{\beta}$ the \textit{ground state}. It is also the unique optimizer (up to symmetry) for the following Gagliardo-Nirenberg inequality
\begin{equation}\label{121}
\|f\|_{L^{2\beta+2}}^{2\beta+2}\leq C^*\big\||D|^{\frac{\beta}{2}}f\big\|^2_{L^2}\|f\|^{2\beta}_{L^2},\;\;\text{for all }f\in H^{\frac{\beta}{2}}.
\end{equation}
Hence, a standard argument shows that if $\|u_0\|_{L^2}<\|Q_{\beta}\|_{L^2}$ then the corresponding solution to \eqref{CP} satisfies
$$\sup_{0\leq t<T}\big\||D|^{\frac{\beta}{2}}u(t)\big\|_{L^2}\leq C(u_0)<+\infty,$$
which implies that the solution is global in time and uniformly bounded in $H^{\frac{\beta}{2}}$. However, unlike the $L^2$-critical Schr\"odinger equation ($\beta=2$), where Dodson \cite{DOD4} proved that the condition $\|u_0\|_{L^2}<\|Q_{\beta}\|_{L^2}$ actually implies scattering at both time direction, in the fractional case, there exists non-scattering solutions (traveling waves) with arbitrarily small mass, due to \cite{KLR,NP} .

\subsection{On the $L^2$-critical NLS problem}Let us give an overview of the results for the $L^2$-critical Schr\"odinger equations:
\begin{equation}
\label{NLS}
\begin{cases}
iu_t+\Delta u+|u|^{\frac{4}{d}}u=0,\quad (t,x)\in\mathbb{R}\times\mathbb{R}^d,\\
u|_{t=0}=u_0\in H^{1}(\mathbb{R}^d).
\end{cases}
\end{equation}

From Weinstein \cite{W1}, we know that for all initial data $u_0\in H^1$ with $\|u_0\|_{L^2}<\|Q\|_{L^2}$, the corresponding solution to \eqref{NLS} is global in time and uniformly bounded in $H^1$. Here $Q$ is called the ground state, which is the unique nonnegative radial $H^1$ solution to the following elliptic equations:
\begin{equation}
\label{116}
\Delta Q-Q+Q^{1+\frac{4}{d}}=0,\quad Q(y)>0,\quad Q\in H^1(\mathbb{R}^d).
\end{equation}
Hence, blow-up for \eqref{NLS} can only occur in the case of $\|u_0\|_{L^2}\geq\|Q\|_{L^2}$.

There are several examples of blow-up solutions for \eqref{NLS}: 

\noindent {\bf(1) Using virial argument}: Let the initial data satisfy $u_0\in H^1$, $xu_0\in L^2$ and 
$$\tilde{E}(u_0)=\frac{1}{2}\int|\nabla u_0|^2-\frac{d}{2d+4}\int|u_0|^{\frac{2d+4}{d}}<0,$$
then the corresponding solution to \eqref{NLS} blows up in finite time.

\noindent {\bf(2) Minimal mass blow-up solutions}: The pseudo-conformal symmetry of \eqref{NLS} yields an explicit minimal blow-up solution:
\begin{equation}
\label{18}
S(t,x)=\frac{1}{|t|^{\frac{d}{2}}}Q\bigg(\frac{x}{t}\bigg)e^{\frac{i}{t}+\frac{i|x|^2}{4t}},
\end{equation}
which blows up at $T=0$ with $\|\nabla S(t)\|_{L^2}\sim 1/|t|$ as $t\rightarrow0$. In \cite{M7}, Merle proved that the solution given by \eqref{18} is the unique finite time blow-up solutions in $H^1$ with critical mass $\|u\|_{L^2}=\|Q\|_{L^2}$, up to symmetries of the equation.

\noindent {\bfseries(3) Log-log blow-up solution}: Numerical simulations \cite{LPSS}, and formal arguments \cite{SS}, suggest
the existence of solutions blowing up like 
$$\|\nabla u(t )\|_{L^2} \sim \sqrt{\frac{\log|\log(T-t )|}{T-t}}$$
 in dimension $d=2$. Perelman \cite{Per} proved the existence of a blow-up solution of this type and its stability in some space $E\subset H^1$. More detailed results have been obtained in a series of papers of Merle and Rapha\"el \cite{MR3,MR2,MR1,MR5,MR4}. They proved the existence of an $H^1$ nonempty open set of initial data leading to finite time blow-up solutions in log-log regime. These solutions behave like a blow-up bubble near the blow-up time:
\begin{equation}
\label{110}
u(t,x)-\frac{1}{\lambda^{\frac{d}{2}}(t)}Q\bigg(\frac{x-x(t)}{\lambda(t)}\bigg)e^{i\gamma(t)}\rightarrow u^*\; {\rm in}\;L^2,\quad {\rm as}\;t\rightarrow T,
\end{equation}
for some parameters $\gamma(t)\rightarrow+\infty$ and $x(t)\rightarrow x(T)$ and the blow-up speed is given by
\begin{equation}
\label{19}
\|\nabla u(t)\|_{L^2}=\frac{\|\nabla Q\|_{L^2}}{\lambda(t)},\quad\lambda(t)=\sqrt{\frac{2\pi(T-t)}{\log|\log(T-t)|}}\big(1+o(1)\big),\;{\rm as}\;t\rightarrow T.
\end{equation}

\noindent {\bfseries(4) Non-generic blow-up solutions}: In \cite{BW}, Bourgain and Wang proved that there exist in dimension $d=1,2$, a family of blow-up solutions with blow-up rate: $\|\nabla u(t)\|_{L^2}\sim 1/(T-t)$, other than the minimal mass blow-up solutions given by \eqref{18}. In \cite{MRS3}, Merle, Rapha\"el and Szeftel proved that such solution is unstable in $H^1$. In ]\cite{KS}, Krieger and Schlag proved the existence of a codimension one manifolds of initial data which leads to blow-up solutions with blow-up rate: $\|\nabla u(t)\|_{L^2}\sim 1/(T-t)$. In \cite{MarR}, Martel and Rapha\"el construct blow-up solutions with blow-up rate: $\|\nabla u(t)\|_{L^2}\sim |\log(T-t)|/(T-t)$.

\subsection{Blow-up results for $L^2$ critical half wave equations}
In the case $\beta=1$, the Cauchy problem \eqref{CP} is called half wave equation. The existence of blow-up solutions in this case is a long standing open problem. 

Unlike the $L^2$-critical NLS \eqref{NLS}, the virial argument does not work in this case. Indeed, we still have:
$$\frac{d}{dt}\bigg(\Im\int x\cdot \nabla u(t)\bar{u}(t)\bigg)=2E(u_0).$$
But the term ${\rm Im}\int x\cdot \nabla u(t)\bar{u}(t)$ cannot be written as the derivative of some nonnegative term. A suitable generalization of the variance term for the half wave equations should be
$$V(t):=\int \big|(-\Delta)^{\frac{1}{4}}(xu(t))\big|^2.$$
However, the appearance of nonlinear terms makes the analysis much more complicated: 
$$V'(t)=2\Im\int x\cdot \nabla u(t)\bar{u}(t)+\text{nonlinear terms}.$$
On the other hand, there is no pseudo-conformal symmetry for half wave equations. We cannot construct minimal mass blow-up solutions directly. However, by a dynamical argument, Krieger, Lenzmann and Rapha\"{e}l \cite{KLR} constructed a minimal mass blow-up solution to the $L^2$-critical half-wave equation with:
$$\big\||D|^{\frac{1}{2}}u(t)\big\|_{L^2}\sim \frac{C}{T-t},\quad as\quad t\rightarrow T,$$
for any given momentum and energy (positive). But unlike the $L^2$ critical Schr\"odinger case, the uniqueness (up to symmetry) for this minimal mass blow-up solution is still not known.
\subsection{Main results}
In this paper, we will construct blow-up solutions for \eqref{CP} for $1\leq \beta<2$. Similar to the half-wave case, the virial argument does not work for general $L^2$-critical fractional Schr\"odinger case either. Besides, we still lack the pseudo-conformal symmetry here to construct minimal mass blow-up solution directly. In \cite{KLR}, the authors claimed that minimal mass blow-up solutions for \eqref{CP} in the case $1<\beta<2$ could also be constructed by using the same argument. In \cite{BHL}, the authors proved that under suitable assumptions, there exist solutions to \eqref{CP}, which blow up either in finite time or infinite time by using a localized virial argument. In this paper, we will focus on the slightly supercritical mass case:
\begin{equation}
\label{111}
u_0\in \mathcal{B}_{\alpha_0}:=\Big\{u_0\in H^{\frac{\beta}{2}}\big|\|Q_{\beta}\|_{L^2}<\|u_0\|_{L^2}<\|Q_{\beta}\|_{L^2}+\alpha_0\Big\}.
\end{equation}

Now, we can state the main result of this paper:
\begin{theorem}[Blow-up dynamics in $\mathcal{B}_{\alpha_0}$]\label{MT1}
There exists $1\leq\beta_0<2$ such that if $\beta_0<\beta<2$, then there exists a universal constant $\alpha_0=\alpha_0(\beta)>0$, such that if $u_0\in \mathcal{B}_{\alpha_0}$, $E(u_0)<0$, then the corresponding solution $u(t)$ to the Cauchy problem \eqref{CP} blows up in finite time $T<\infty$, with the following upper bound on the blow-up rate:
\begin{equation}
\label{114}
\big\||D|^{\frac{\beta}{2}}u(t)\big\|_{L^2}\leq C^*\sqrt{\frac{|\log(T-t)|^{\frac{1}{8}}}{T-t}},\quad \text{as }t\rightarrow T,
\end{equation}
for some universal constant $C^*>0$ independent of $\beta$.
\end{theorem}

\begin{remark}
The set of initial data satisfying the conditions mentioned in Theorem \ref{MT1} is not empty. Since the ground state has zero energy $E(Q_{\beta})=0$, one may consider $u_{0,\delta}=(1+\delta)Q_{\beta}$ with $0<\delta\ll1$. The proof of Theorem \ref{MT1} is similar to that in \cite{MR3,MR2,MR1,MR4}, but due to a lack of Galilean transform and the nonlocal structure, we need some additional arguments.
\end{remark}

\begin{remark}\label{R1}
The reason why we assume that $\beta$ is close to $2$ is that we need some spectral property related to the linearization of the ground state $Q_\beta$. The proof of such property relies on a perturbation argument and the continuity of $Q_\beta$ with respect to $\beta$, which justifies the assumption on $\beta$. We mention here the analysis%
\footnote{See Section \ref{S2}--\ref{S4}.}
 in this paper can be extended to all $1\leq\beta<2$ if such spectral property holds true.
\end{remark}

\begin{remark}
Theorem \ref{MT1} gives a first construction of finite time blow-up solutions with supercritical mass for $L^2$ critical fractional NLS \eqref{CP}. This solution is also stable in the sense that the set of initial data satisfying the conditions mentioned in Theorem \ref{MT1} is an open subset in $H^{\frac{\beta}{2}}$. This is completely different from the minimal mass blow-up solution constructed in \cite{KLR} since the minimal blow-up solution is obviously unstable in $H^{\frac{\beta}{2}}$. On the other hand,  Boulenger, Himmelsbach and Lenzmann proved in \cite{BHL} that in higher dimension case, radial solutions with negative energy to the $L^2$-critical fractional NLS will blow up in finite time or infinite time. Theorem \ref{MT1} gives an example of finite time blow-up solutions in this case. 
\end{remark}

\begin{remark}
The scaling structure of \eqref{CP} gives the following {\it a priori} lower bound for all finite time blow-up solutions:
$$\big\||D|^{\frac{\beta}{2}}u(t)\big\|_{L^2}\gtrsim (T-t)^{-\frac{\beta}{4}}.$$
But nontrivial lower bound of blow-up rate in this case is still not known.
\end{remark}

\begin{remark}

The upper bound in Theorem \ref{MT1} is not likely to be sharp. Indeed this upper bound is related to the approximation of the nonlinear blow-up profile $W_{b,v}$ constructed in Section \ref{S2}. A sharp estimate of the error term is needed to obtain sharp upper bound of the blow-up rate, which seems difficult in the nonlocal case. On the other hand, one may see from the construction of $W_{b,v}$ that the blow-up profile is approximated by its Taylor's expansion with respect to $b$ and $v$. In this paper, we stop at the order of $b^4$. This is because the equation at the next order is significantly more complicated, which seems hard to verify. Besides, this type of improvement is also not likely to provide a sharp upper bound on the blow-up rate. Since, numerical simulation \cite{KSM} suggests that the exact blow-up rate for the blow-up solutions introduced in Theorem \ref{MT1} should be
$$\big\||D|^{\frac{\beta}{2}}u(t)\big\|_{L^2} \sim \sqrt{\frac{\log|\log(T-t)|}{T-t}},$$
the same as the local case when $\beta=2$. This requires a sharp approximation of the blow-up profile where the error term is controlled by $e^{-C/|b|}$, which at this moment is not known.

\end{remark}

\begin{remark}
For initial data with slightly supercritical mass and zero energy, by slightly modifying the proof of Theorem \ref{MT1}, we can prove that either the solution $u(t)$ itself or%
\footnote{We mention here that this $v(t)$ is still a solution of \eqref{CP}.}
$v(t):=\bar{u}(-t)$ blows up in finite time with the same estimate on the upper bound of the blow-up rate. In the local case when $\beta=2$, Merle and Rapha\"el \cite[Theorem 6]{MR2} shows that if we additionally assume that $xu_0\in L^2$, then the zero energy solution $u(t)$ will blow up in both time direction. However, this type of result is still not known for the nonlocal case when $1<\beta<2$, since we lack here a uniqueness result for the minimal mass blow-up solution constructed in \cite{KLR} as well as a clear characterization of blow-up solutions in the virial space.
\end{remark}

\begin{remark}
As a similar model, the $L^2$-critical generalized Benjamin-Ono equation:
$$\partial_t-|D|^\beta\partial_xu+(u|u|^{2\beta})_x=0,$$
with $1\leq \beta<2$, has also received a lot of attraction recently. The local case when $\beta=2$ corresponds to the $L^2$-critical gKdV equation where the theory of formation of singularity has been well studied in a series work of Martel and Merle \cite{MM1,MM2,MM4,MM5,MM3} and Martel, Merle and Rapha\"el \cite{MMR1,MMR2,MMR3}. While for the nonlocal case, when $\beta=1$, this equation becomes the modified Benjamin-Ono equation, where the existence of minimal mass blow-up solution has been proved by Martel and Pilod \cite{MP}, similar as what Krieger, Lenzmann and Rapha\"el \cite{KLR} did for the $L^2$-critical half-wave equation. We expect similar result as Theorem \ref{MT1} can be proved for the $L^2$-critical generalized Benjamin-Ono equation, since the blow-up results in the local case (gKdV) is known and is similar to the local case of \eqref{CP}. However, this type of results is still mostly open.
\end{remark}

\subsection{Outline of the proof}
The basic idea of this paper is similar to that of \cite{MR3,MR2,MR1,MR4}. We first notice that the solution is close to the ground state up to scaling, translation and phase, due to the assumption of slightly supercritical mass and negative energy. Then we can linearize the solution at the ground state so that the Cauchy problem \eqref{CP} can be viewed as an ODE system of the well-chosen parameters and a nonlinear PDE of an error term. We hope to find a suitable control of the error term such that the error does not perturb the ODE system. Therefore the behavior of solutions to the ODE system can fully describe the behavior of the original solution. More precisely, we have the following steps:

\subsubsection{The nonlinear blow-up profile}
We are seeking for solution of the following form
\begin{align*}
&u(t,x)=\frac{1}{\lambda^{\frac{1}{2}}(t)}W_{b(t),v(t)}\bigg(\frac{x-x(t)}{\lambda(t)}\bigg)e^{i\gamma(t)},\\
&\frac{ds}{dt}=\frac{1}{\lambda^{\beta}},\;\frac{\lambda_s}{\lambda}+b=0,\;\frac{x_s}{\lambda}=v,\;v_s+bv=0,\;\gamma_s+1=0,\;b_s=0,
\end{align*}
which after a direct computation leads to the following equation for $W_{b,v}$:
$$-\Psi_{b,v}:=ib\Lambda W_{b,v}-ivW'_{b,v}-ibv\partial_vW_{b,v}-|D|^{\frac{\beta}{2}}W_{b,v}-W_{b,v}+|W_{b,v}|^{2\beta}W_{b,v}=0.$$
Using the properties of the linearized operator $\pmb{L}^{\beta}$ at $Q_{\beta}$, we may find a suitable approximate solution $W_{b,v}$, such that
$$|\Psi_{b,v}|\lesssim |b|^{5}+v^{2}.$$

\subsubsection{Geometrical decomposition and modulation theory}
Under the assumption of slightly supercritical mass and negative energy, we may use a standard implicit function argument to write the solution in the following form
\begin{equation}
\label{118}
u(t,x)=\frac{1}{\lambda^{\frac{1}{2}}(t)}\big[W_{b(t),v(t)}+\varepsilon(t)\big]\bigg(\frac{x-x(t)}{\lambda(t)}\bigg)e^{i\gamma(t)},
\end{equation}
where the complex valued error term $\varepsilon=\varepsilon_1+i\varepsilon_2$ satisfies some suitable orthogonality conditions. Here we introduce the velocity parameter $v$ to deal with the lack of Galilean transform. We will see the velocity parameter asymptotically vanishes sufficiently fast so that it does not perturb the system. We also mentioned here that the choice of the orthogonality conditions implies the following relation:
\begin{equation}\label{117}
b\sim (\varepsilon_2,\Lambda Q_{\beta}).
\end{equation}

Differentiating those orthogonality conditions, we have a first control of the parameters:
\begin{align}
&\bigg|\frac{\lambda_s}{\lambda}+b\bigg|+|b_s|+|v_s+bv|+\bigg|\gamma_s+1+\frac{1}{\|\Lambda Q_{\beta}\|_{L^2}^2}(\varepsilon_1,L_+^{\beta}\Lambda^2Q_{\beta})\bigg|+\bigg|\frac{x_s}{\lambda}-v\bigg|\nonumber\\
&\lesssim \delta(\alpha_0)\bigg(\int\big||D|^{\frac{\beta}{2}}\varepsilon\big|^2+\int|\varepsilon|^2e^{-|y|}\bigg)^{\frac{1}{2}}+|b|^5+|v|^2+\lambda^{\beta}|E(u_0)|.\label{120}
\end{align}
Our next task is to find a suitable control of the error term $\varepsilon$.

\subsubsection{The local virial estimate and the spectral property} An important feature for \eqref{CP} is that it still has the following virial identity:
$$\frac{d}{dt}\bigg(\Im\int x\cdot \nabla u(t)\bar{u}(t)\bigg)=2E(u_0).$$
Although we cannot directly use the virel identity to construct blow-up solutions, we may still use it to get a suitable control of the error term $\varepsilon$. More precisely, if we inject the geometrical decomposition \eqref{118} into the virel identity using \eqref{117}, we have the following local virel estimate:
\begin{equation}
\label{119}
b_s\gtrsim \widetilde{H}^{\beta}(\varepsilon,\varepsilon)-\lambda^{\beta}E(u_0)+v^2-|b|^{10}-\delta(\alpha_0)\bigg(\int\big||D|^{\frac{\beta}{2}}\varepsilon\big|^2+\int|\varepsilon|^2e^{-|y|}\bigg),
\end{equation}
where%
\footnote{See \eqref{481} for the detailed definition of $\widetilde{H}^\beta$.} 
$$\widetilde{H}^{\beta}(\varepsilon,\varepsilon)=\int\big||D|^{\frac{\beta}{2}}\varepsilon\big|^2+\int V_1(y)\varepsilon_1^2+\int V_2(y)\varepsilon_2^2+(\varepsilon_1,V_3)(\varepsilon_2,V_4)$$
is some quadratic form for $\varepsilon=\varepsilon_1+i\varepsilon_2\in \widetilde{H}^{\frac{\beta}{2}}$. Thanks to the negative energy assumption and the spectral property (coercivity of $\widetilde{H}^{\beta}(\varepsilon,\varepsilon)$), we have for all $s\in[0,+\infty)$,
$$b_s\geq -C|b|^{10}.$$

\subsubsection{Proof of Theorem \ref{MT1}}
From \eqref{120}, \eqref{119} and the spectral property, we can show that 
$$\int\big||D|^{\frac{\beta}{2}}\varepsilon\big|^2+\int|\varepsilon|^2e^{-|y|}\ll |b|,$$
in a time average sense. Using \eqref{120} and \eqref{119} again, we have the following two key points:
\begin{enumerate}
\item There exists a $s_0>0$, such that for all $s\geq s_0$, we have $b(s)>0$.

\item For all $s\geq s_0$, we have $\lambda_s/\lambda\sim -b$ in a time average sense.
\end{enumerate}

Note that the the local virial estimate implies that $b_s\geq -Cb^{10}$ for all $s\geq s_0$. After integration, we have for all $s$ large enough,
$$b(s)\geq \frac{C}{s^{\frac{1}{9}}},\quad \lambda(s)\geq e^{-Cs^{\frac{8}{9}}},$$
and for all $s_2>s_1$ large enough, 
$$\lambda(s_2)\geq \frac{1}{2}\lambda(s_1).$$
These estimates remove the possibility that the $H^{\frac{\beta}{2}}$ norm of the solution oscillates in time, which forces the solution to blow up in finite time $T<+\infty$. Finally, after a change of coordinate, the above estimates also imply that
$$\lambda(t)\geq \bigg(\frac{T-t}{|\log(T-t)|^{\frac{1}{8}}}\bigg)^{\frac{1}{\beta}}.$$
Now the proof of Theorem \ref{MT1} is finished.

\subsection{Notations}\label{S1}
We use $|D|^s$ to denote the fractional order derivatives for $s\geq0$. That is
$$
\widehat{|D|^su}(\xi)=|\xi|^{s}\hat{u}(\xi).
$$
We use
$$(f,g)=\int \bar{f}g$$
as the inner product on $L^2(\mathbb{R};\mathbb{C})$.

We denote by $Q_{\beta}$ the ground state of \eqref{CP}, which is the unique radial nonnegative $H^{\frac{\beta}{2}}$ solution of 
$$|D|^{\beta}Q_{\beta}+Q_{\beta}-|Q_{\beta}|^{2\beta}Q_{\beta}=0.$$
In the local case when $\beta=2$, we also denote by $Q=Q_{2}$, which has an explicit expression in dimension one:
$$Q(x)=\bigg(\frac{3}{\cosh^2(2x)}\bigg)^{\frac{1}{4}}.$$

For a regular enough function $f$, we define the $L^2$-critical scaling operator as following:
$$\Lambda f(y):=-\frac{d}{d\lambda}\bigg|_{\lambda=1}\frac{1}{\lambda^{\frac{1}{2}}}f\bigg(\frac{y}{\lambda}\bigg)=\frac{1}{2}f+yf'.$$
We also write $\Lambda^k f$ for $k\in\mathbb{N}$ for the iterates of $\Lambda$.

In some parts of the paper, we will identify any complex valued function $f:\mathbb{R}\rightarrow\mathbb{C}$ with $\pmb{f}:\mathbb{R}\rightarrow \mathbb{R}^2$ in the following sense:
$$\pmb{f}=
\begin{bmatrix}
\Re f\\
\Im f
\end{bmatrix}.
$$
We still use the notion $\Re\pmb{f}$ and $\Im\pmb{f}$ for the real and imaginary part of the function $\pmb{f}$.

We also identify the multiplication by $i$ in $\mathbb{C}$ with the multiplication by the following matrix
$$\pmb{J}=
\begin{bmatrix}
0&-1\\
1&0
\end{bmatrix}.
$$

We denote by $\pmb{f}\cdot\pmb{g}$ the standard multiplication in $\mathbb{C}$,
$$\pmb{f}\cdot\pmb{g}=
\begin{bmatrix}
\Re f\Re g-\Im f\Im g\\
\Re f\Im g+\Im f\Re g
\end{bmatrix},
$$
and $\bar{\pmb{f}}$ the complex conjugate of $\pmb{f}$,
$$\bar{\pmb{f}}=\begin{bmatrix}
\Re f\\
-\Im f
\end{bmatrix}.
$$

We also denote the linearized operator at the ground state $Q_{\beta}$ by
$$\pmb{L}^{\beta}=
\begin{bmatrix}
L_+^{\beta}&0\\
0&L_-^{\beta}
\end{bmatrix},
$$
with the following two scalar operators
$$L_+^{\beta}=|D|^{\frac{\beta}{2}}+1-(2\beta+1)Q_{\beta}^{2\beta},\quad L_-^{\beta}=|D|^{\frac{\beta}{2}}+1-Q_{\beta}^{2\beta},$$
acting on $L^2(\mathbb{R};\mathbb{R})$.

Finally, we denote by $\delta(\alpha)>0$ a small universal constant such that 
$$\lim_{\alpha\rightarrow0^+}\delta(\alpha)=0.$$

\section{Construction of the nonlinear blow-up profile}\label{S2}
This section is devoted to construct an approximate blow-up profile $W_{b,v}$ for the following equation:
$$ib\Lambda W_{b,v}-ivW'_{b,v}-ibv\partial_vW_{b,v}-|D|^{\frac{\beta}{2}}W_{b,v}-W_{b,v}+|W_{b,v}|^{2\beta}W_{b,v}=0.$$

For convenience, we identify a complex valued function $f:\mathbb{R}\rightarrow\mathbb{C}$ with the vector valued function $\pmb{f}:\mathbb{R}\rightarrow \mathbb{R}^2$, as we mentioned in Section \ref{S1}.

Now, we can state our result for the approximate blow-up profile.
\begin{proposition}[Approximate blow-up profile]\label{P1}
For all $2>\beta>1$, there exist constants $b_0>0$, $v_0>0$, such that if $|b|<b_0$, $|v|<v_0$, then there exists a smooth function of the following form
\begin{equation}
\label{21}
\pmb{W}_{b,v}=\pmb{Q}_{\beta}+\sum_{j=1}^4b^{j}\pmb{R}_{j,0}+v\pmb{R}_{0,1}+v^2\pmb{R}_{0,2}+bv\pmb{R}_{1,1}+b^2v\pmb{R}_{2,1},
\end{equation}
satisfying the following:
\begin{enumerate}
\item The equation of $\pmb{W}_{b,v}$: the profile $\pmb{W}_{b,v}$ satisfies
\begin{equation}
\label{22}
-\pmb{\Psi}_{b,v}=\pmb{J}b\Lambda \pmb{W}_{b,v}-\pmb{J}v\pmb{W}'_{b,v}-\pmb{J}bv\partial_v\pmb{W}_{b,v}-|D|^{\frac{\beta}{2}}\pmb{W}_{b,v}-\pmb{W}_{b,v}+|\pmb{W}_{b,v}|^{2\beta}\pmb{W}_{b,v}.
\end{equation}
Here the error term $\pmb{\Psi}_{b,v}$ satisfies the following estimates:
\begin{align}
&\|\pmb{\Psi}_{b,v}\|_{H^m}\lesssim_{m}|b|^5+v^2(|v|+|b|),\label{23}\\
&\|\langle x\rangle^{(1+\beta)}\Lambda^n\pmb{\Psi}_{b,v}(x)\|_{L^{\infty}}\lesssim_{n}|b|^5+v^2(|v|+|b|).\label{219}
\end{align}
\item Regularity and decay estimate:  the functions $\{\pmb{R}_{k,\ell}\}_{0\leq k\leq 4,0\leq\ell\leq2}$ satisfy the following estimates
\begin{equation}
\label{24}
\|\Lambda^n \pmb{R}_{k,\ell}\|_{H^m}+\|\langle x\rangle^{(1+\beta)}\Lambda^n\pmb{R}_{k,\ell}(x)\|_{L^{\infty}}\lesssim_{m,n}1,
\end{equation}
for all $m\in\mathbb{N}$ and $n=0,1,2$.
\item Degeneracy of the energy:
\begin{equation}\label{240}
\beta E(\pmb{W}_{b,v})=c_0v^2+O\big(|b|^5+v^2(|v|+|b|)\big),
\end{equation}
where $c_0>0$ is a positive constant.
\item The scaling invariance: there holds
\begin{align}
\label{246}
&|D|^\beta (\Lambda \pmb{W}_{b,v})+\Lambda \pmb{W}_{b,v}-|\pmb{W}_{b,v}|^{2\beta}(\Lambda \pmb{W}_{b,v})-2\beta \pmb{W}_{b,v}|\pmb{W}_{b,v}|^{2\beta-2}\Re(\pmb{W}_{b,v}\Lambda\overline{\pmb{W}}_{b,v})\nonumber\\
&=\beta(\pmb{\Psi}_{b,v}+\pmb{J}b\Lambda \pmb{W}_{b,v}-\pmb{J}v\nabla \pmb{W}_{b,v}-\pmb{J}bv\partial_v \pmb{W}_{b,v}-\pmb{W}_{b,v})\nonumber\\
&\quad-\Lambda\pmb{\Psi}_{b,v}-\pmb{J}b\Lambda^2\pmb{W}_{b,v}+\pmb{J}v\Lambda(\nabla \pmb{W}_{b,v})+\pmb{J}bv\Lambda(\partial_v\pmb{W}_{b,v}).
\end{align}
\end{enumerate}
\end{proposition}

\begin{remark}
For simplicity, we omit the dependence of $\beta$ here.
\end{remark}
\begin{remark}
The proof of Proposition \ref{P1} is similar to \cite[Proposition 4.1]{KLR}, but due to the different objects we are dealing with, the nonlinear profile is slightly different.
\end{remark}
\begin{remark}\label{25}
In the proof of Proposition \ref{P1}, we actually show that the functions $\pmb{R}_{k,\ell}$ have the following structure:
\begin{align}
&\pmb{R}_{1,0}=
\begin{bmatrix}
0\\
\mathrm{even}
\end{bmatrix},
\;
\pmb{R}_{2,0}=
\begin{bmatrix}
\mathrm{even}\\
0
\end{bmatrix},
\;
\pmb{R}_{3,0}=
\begin{bmatrix}
0\\
\mathrm{even}
\end{bmatrix},
\;
\pmb{R}_{4,0}=
\begin{bmatrix}
\mathrm{even}\\
0
\end{bmatrix},\nonumber\\
&\pmb{R}_{0,1}=
\begin{bmatrix}
0\\
\mathrm{odd}
\end{bmatrix},
\;
\pmb{R}_{0,2}=
\begin{bmatrix}
\mathrm{even}\\
0
\end{bmatrix},
\;
\pmb{R}_{1,1}=
\begin{bmatrix}
\mathrm{odd}\\
0
\end{bmatrix},
\;
\pmb{R}_{2,1}=
\begin{bmatrix}
0\\
\mathrm{odd}
\end{bmatrix}.
\end{align}
Here the words ``even" and ``odd" mean that the real or imaginary part of $\pmb{R}_{k,\ell}$ is an even or odd function.
\end{remark}

\begin{remark}
We will see that the estimate on the error term $\pmb{\Psi}_{b,v}$ determines the upper bound of the blow-up rate. We expect the estimate can be improved to $e^{-C/|b|}$, which will lead to a log-log upper bound on the blow-up rate as the numeric simulation \cite{KSM} predicted. But this is still not known.

On the other hand, from the proof of Proposition \ref{P1}, one may expand $W_{b,v}$ to higher order to obtain a better approximation. But this does not improve the result of this paper too much, since the new result is also not likely to be sharp. Besides, the equation at the next order of $b$ is significantly more complicated than the equation at the order of $|b|^3$ and $|b|^4$. Although the author believes it is true, checking them will involve a lot more technical difficulties.
\end{remark}

The idea to construct such a blow-up profile is to write down the equation of each $\pmb{R}_{b,v}$. We will see that the solvability of these equations follows from the invertibility of the linearized operator $\pmb{L}^{\beta}$. We will give a list of properties of $\pmb{L}^{\beta}$ that will be used in this paper. Their proofs are either direct or can be found in \cite{AT,FL,FLS} and \cite[Appendix A]{KLR}.

\begin{lemma}[Properties of the linearized operator]
\label{L1}
The linearized operator 
$$\pmb{L}^{\beta}=
\begin{bmatrix}
L_+^{\beta}&0\\
0&L_-^{\beta}
\end{bmatrix}
$$
has the following properties:
\begin{enumerate}
\item Kernel%
\footnote{Recall that the scalar operators $L^{\beta}_+$ and $L^{\beta}_-$ are acting on $L^2(\mathbb{R};\mathbb{R})$, while the operator $\pmb{L}^{\beta}$ is acting on $L^2(\mathbb{R};\mathbb{C})$.}%
: $\ker L^{\beta}_{+}=\mathrm{span}\{\nabla Q_{\beta}\}$, $\ker L^{\beta}_{-}=\mathrm{span}\{Q_{\beta}\}$, and 
\begin{equation}
\label{26}
\ker \pmb{L}^{\beta}=\mathrm{span} \bigg\{
\begin{bmatrix}
\nabla Q_{\beta}\\
0
\end{bmatrix},
\begin{bmatrix}
0\\
Q_{\beta}
\end{bmatrix}
\bigg\}.
\end{equation}
For all $f_1,f_2\in L^2(\mathbb{R};\mathbb{R})$, if $(f_1,\nabla Q_{\beta})=(f_2,Q_{\beta})=0$, then there exist unique $g_1,g_2\in H^{\beta}(\mathbb{R};\mathbb{R})$ such that $L_+^{\beta}g_1=f_1$ and $L^{\beta}_-g_2=f_2$.
\item Eigenvalue: $L^{\beta}_+$ has exactly two eigenvalues: $0$ and a negative one $\kappa_0<0$. The negative eigenvalue $\kappa_0$ is associated to an even positive eigenfunction $\chi_0$.
\item Coercivity: For all $\varepsilon_1,\varepsilon_2\in H^{\frac{\beta}{2}}(\mathbb{R};\mathbb{R})$, if $(\varepsilon_1,\nabla Q_{\beta})=(\varepsilon_1,\chi_0)=(\varepsilon_2,Q_{\beta})=0$, then we have
\begin{equation}
\label{27}
(L_+^{\beta}\varepsilon_1,\varepsilon_1)\geq (\varepsilon_1,\varepsilon_1),\quad (L_-^{\beta}\varepsilon_2,\varepsilon_2)\geq (\varepsilon_2,\varepsilon_2).
\end{equation}
\item Scaling rule:
\begin{equation}
\label{29}
L^{\beta}_+(\Lambda Q_{\beta})=-\beta Q_{\beta},\quad 
\end{equation}
\item Invertibility: For all $f,g\in H^k(\mathbb{R};\mathbb{R})$ and $k\in\mathbb{N}$, suppose $(f,\nabla Q_{\beta})=(g,Q_{\beta})=0$, then we have 
\begin{equation}\label{28}
\|(L^{\beta}_+)^{-1}f\|_{H^{k+\beta}}\lesssim_k\|f\|_{H^k},\quad \|(L^{\beta}_-)^{-1}g\|_{H^{k+\beta}}\lesssim_k\|g\|_{H^k},
\end{equation}
and the decay estimates
\begin{align}
&\|\langle x\rangle^{1+\beta}(L^{\beta}_+)^{-1}f\|_{L^{\infty}}\lesssim \|\langle x\rangle^{1+\beta}f\|_{L^{\infty}},\label{210}\\
&\|\langle x\rangle^{1+\beta}(L^{\beta}_-)^{-1}g\|_{L^{\infty}}\lesssim \|\langle x\rangle^{1+\beta}g\|_{L^{\infty}}\label{211}.
\end{align}
\end{enumerate}
\end{lemma}

We now turn to the proof of Proposition \ref{P1}.  As mentioned before, the idea is to solve \eqref{21} order by order. Here in the nonlocal case when $1<\beta<2$, the nonlinear term $|\pmb{W}_{b,v}|^{2\beta}\pmb{W}_{b,v}$ is generally not smooth with respect to $b$ and $v$, when $|b|$ and $|v|$ are small enough. But we have 
$$\big(|\pmb{W}_{b,v}|^{2\beta}\pmb{W}_{b,v}\big)\Big|_{b=0,v=0}=
\begin{bmatrix}
Q_{\beta}^{2\beta+1}
\\0
\end{bmatrix}.
$$
Here the ground state $Q_{\beta}$ satisfies
$$Q_{\beta}(y)\gtrsim \frac{1}{\langle y\rangle^{1+\beta}},$$
for all $y\in\mathbb{R}$. Hence if the decay estimate \eqref{24} holds true, then for small enough $b$ and $v$, the nonlinear term $|\pmb{W}_{b,v}|^{2\beta}\pmb{W}_{b,v}$ is infinitely differentiable with respect to $b$ and $v$, which makes it possible to solve each $\pmb{R}_{k,\ell}$ order by order

\begin{proof}[Proof of Proposition \ref{P1}]As mentioned before, we will prove Proposition \ref{P1} in the following steps.

\subsubsection*{\bf Step 1: Solving $\pmb{R}_{k,\ell}$ formally.} We first solve each $\pmb{R}_{k,\ell}$ order by order.

\textit{Order $\mathcal{O}(1)$}: It is easy to see that $\pmb{Q}_{\beta}=[Q_{\beta},0]^{\top}$ is what we need.

\textit{Order $\mathcal{O}(b)$}: A standard computation shows that $\pmb{R}_{1,0}$ satisfies
\begin{equation}
\label{214}
\pmb{L}^{\beta}\pmb{R}_{1,0}=\pmb{J}\Lambda\pmb{Q}_{\beta}.
\end{equation}
We note that $\pmb{J}\Lambda\pmb{Q}_{\beta}=[0,\Lambda Q_{\beta}]^{\top}$ satisfying $\pmb{J}\Lambda\pmb{Q}_{\beta}\perp\ker \pmb{L}^{\beta}$, since $(\Lambda Q_{\beta},Q_{\beta})=0$. Hence we may chose
\begin{equation}
\label{215}
\pmb{R}_{1,0}=\begin{bmatrix}
0\\S_1
\end{bmatrix},
\end{equation}
where $S_1=(L^{\beta}_-)^{-1}\Lambda Q_{\beta}$.

\textit{Order $\mathcal{O}(v)$}: For $\pmb{R}_{0,1}$, we have the following equation
\begin{equation}
\label{216}
\pmb{L}^{\beta}\pmb{R}_{0,1}=-\pmb{J}\nabla \pmb{Q}_{\beta}.
\end{equation}
It is also easy to check that $\pmb{J}\nabla \pmb{Q}_{\beta}=[0,\nabla Q_{\beta}]^\top$ satisfies $\pmb{J}\nabla \pmb{Q}_{\beta}\perp\ker\pmb{L}^{\beta}$, since $Q_{\beta}$ is an even function while $\nabla Q_{\beta}$ is an odd function. Hence we may choose
\begin{equation}
\label{217}
\pmb{R}_{0,1}=\begin{bmatrix}
0\\
G_1
\end{bmatrix},
\end{equation}
where $G_1=(-L^{\beta}_-)^{-1}\nabla Q_{\beta}$.

\textit{Order $\mathcal{O}(bv)$}: For $\pmb{R}_{1,1}$, we note that $\Re\pmb{R}_{1,0}=\Re\pmb{R}_{0,1}=\Im\pmb{Q}_{\beta}=0$. Hence, we obtain the following equation for $\pmb{R}_{1,1}$:
\begin{equation}
\label{218}
\pmb{L}^{\beta}\pmb{R}_{1,1}=-\pmb{J}\pmb{R}_{0,1}+\pmb{J}\Lambda \pmb{R}_{0,1}-\pmb{J}\nabla \pmb{R}_{1,0}+2\beta\Re(\pmb{R}_{1,0}\cdot\bar{\pmb{R}}_{0,1})\pmb{Q}_{\beta}^{2\beta-1}
\end{equation}
To find such $\pmb{R}_{1,1}$, we need to check that the right hand side of \eqref{218} $\perp\ker\pmb{L}^{\beta}$, which is equivalent to
\begin{equation}
\label{220}
(G_1-\Lambda G_1+\nabla S_1+2\beta S_1G_1Q_{\beta}^{2\beta-1},\nabla Q_{\beta})=0.
\end{equation}
Using the standard commutator formula $[\Lambda,\nabla]=-\nabla$, we have
\begin{align}
\label{221}
-(\nabla Q_{\beta},\Lambda G_1)&=(\Lambda \nabla Q_{\beta},G_1)=(\nabla\Lambda Q_{\beta},G_1)-(\nabla Q_{\beta},G_1)\nonumber\\
&=(\nabla L_-^{\beta}S_1,G_1)-(\nabla Q_{\beta},G_1).
\end{align}
We also have
\begin{align}
\label{222}
&(\nabla L_-^{\beta}S_1,G_1)+(\nabla Q_{\beta},\Lambda G_1)=-(S_1,L_-^{\beta}\nabla G_1)+(\nabla L_-^{\beta} G_1,S_1)\nonumber\\
&=(S_1, [\nabla,L_-^{\beta}]G_1)=-(S_1,\nabla (Q_{\beta}^{2\beta})G_1)=-2\beta(\nabla Q_{\beta},S_1G_1Q_{\beta}^{2\beta-1})
\end{align}
Combining \eqref{221} and \eqref{222}, we obtain \eqref{220}. Hence, there exist a unique $\pmb{R}_{1,1}\perp\ker \pmb{L}^{\beta}$ satisfying \eqref{218}. Since, $Q_{\beta}$ and $S_1$ are even functions and $G_1$ is odd, we have
\begin{equation}
\label{223}
\pmb{R}_{1,1}=\begin{bmatrix}
F_{1}\\
0
\end{bmatrix},
\end{equation} 
with some odd function $F_{1}$.

\textit{Order $\mathcal{O}(b^2)$}: We have the following equation for $\pmb{R}_{2,0}$:
\begin{equation}
\label{224}
\pmb{L}^{\beta}\pmb{R}_{2,0}=\pmb{J}\Lambda\pmb{R}_{1,0}+\beta|\pmb{R}_{1,0}|^2\pmb{Q}_{\beta}^{2\beta-1}.
\end{equation}
Since $\pmb{R}_{1,0}=[0,S_1]^\top$, the solvability condition for $\pmb{R}_{2,0}$ is equivalent to
\begin{equation}
\label{225}
-(\nabla Q_{\beta},\Lambda S_1)+(\nabla Q_{\beta},\beta S_1^2Q_{\beta}^{2\beta-1})=0
\end{equation}
Since, $S_1$ and $Q_{\beta}$ are both even functions, the above condition is obviously true. Hence, we have
\begin{equation}
\label{226}
\pmb{R}_{2,0}=\begin{bmatrix}
S_2\\
0
\end{bmatrix},
\end{equation}
with some even function $S_2=(L_+^{\beta})^{-1}(-\Lambda S_1+\beta S_1^2Q_{\beta}^{2\beta-1})$.

\textit{Order $\mathcal{O}(v^2)$}: We have the following equation for $\pmb{R}_{0,2}$
\begin{equation}
\label{227}
\pmb{L}^{\beta}\pmb{R}_{0,2}=-\pmb{J}\nabla \pmb{R}_{0,1}+\beta|\pmb{R}_{0,1}|^2\pmb{Q}_{\beta}^{2\beta-1}.
\end{equation}
Since $\pmb{R}_{0,1}=[0,G_1]^\top$, the solvability condition reduces to
\begin{equation}
\label{228}
(\nabla Q_{\beta},\nabla G_1)+(\nabla Q_{\beta},G_1^2Q_{\beta}^{2\beta-1})=0.
\end{equation}
This condition clearly holds true, since $G_1$ is odd and $Q_{\beta}$ is even. Hence there exist a unique $\pmb{R}_{0,2}$ satisfying \eqref{227} with
\begin{equation}
\label{229}
\pmb{R}_{0,2}=\begin{bmatrix}
G_2\\
0
\end{bmatrix},
\end{equation}
where $G_2=(L_+^{\beta})^{-1}(\nabla G_1+G_1^2Q_{\beta}^{2\beta-1})$ is an even function.

\textit{Order $\mathcal{O}(b^3)$}:  We notice that $\Re \pmb{R}_{1,0}=\Im \pmb{R}_{2,0}=0$, so we have the following equation for $\pmb{R}_{3,0}$,
\begin{equation}
\label{230}
\pmb{L}^{\beta}\pmb{R}_{3,0}=\pmb{J}\Lambda \pmb{R}_{2,0}+\beta Q^{2\beta-2}_{\beta}\big(2\pmb{Q}_{\beta}\cdot\pmb{R}_{2,0}+|\pmb{R}_{1,0}|^2\big)\pmb{R}_{1,0}.
\end{equation}
The solvability of \eqref{230} is equivalent to
\begin{equation}
(Q_{\beta},\Lambda S_2)+(Q_{\beta},2\beta S_1S_2Q_{\beta}^{2\beta-1})+(Q_{\beta},\beta Q_{\beta}^{2\beta-2}S_1^3)=0.
\end{equation}
Recall that 
$$L_-^{\beta}S_1=\Lambda Q_{\beta},\quad L_+^{\beta}S_2=-\Lambda S_1+\beta S_1^2Q_{\beta}^{2\beta-1}.$$
We have
\begin{align*}
&(Q_{\beta},\Lambda S_2)+(Q_{\beta},2\beta S_1S_2Q_{\beta}^{2\beta-1})+(Q_{\beta},\beta Q_{\beta}^{2\beta-2}S_1^3)\\
=&-(\Lambda Q_{\beta},S_2)+2\beta(S_2,Q_{\beta}^{2\beta}S_1)+\beta(Q_{\beta}^{2\beta-1},S_1^3)=-(L_+^{\beta}S_1,S_2)+\beta(Q_{\beta}^{2\beta-1},S_1^3)\\
=&-(S_1,-\Lambda S_1+\beta S_1^2Q_{\beta}^{2\beta-1})+\beta(Q_{\beta}^{2\beta-1},S_1^3)=0.
\end{align*}
Hence, there exists a unique $\pmb{R}_{3,0}$ satisfying \eqref{230} with
\begin{equation}
\label{231}
\pmb{R}_{3,0}=\begin{bmatrix}
0\\
S_3
\end{bmatrix},
\end{equation}
with some even function $S_3=(L_-^\beta)^{-1}(-\Lambda S_2+2\beta S_1S_2Q_{\beta}^{2\beta-1}+\beta S_1^3Q_{\beta}^{2\beta-2})$.

\textit{Order $\mathcal{O}(b^4)$}: We notice that $\Re \pmb{R}_{1,0}=\Im \pmb{R}_{2,0}=\Re \pmb{R}_{3,0}=0$, hence
\begin{align*}
|\pmb{W}_{b,v}|^{2\beta}\Big|_{v=0}&=|(Q_{\beta}+b^2S_2+b^4\Re\pmb{R}_{4,0})^2+b^2(S_1+b^2S_3+b^3\Im \pmb{R}_{4,0})^2|^{\beta}:=|F(b^2)|^{\beta}.
\end{align*}
Since $F\in C^3$ with $F(0)=Q^2_{\beta}>0$, considering the Taylor's expansion of $|F(b^2)|^\beta$ with respect to $b^2$ at $b^2=0$, we have
\begin{align*}
|\pmb{W}_{b,v}|^{2\beta}\Big|_{v=0}&=Q_{\beta}^{2\beta}+\beta Q_{\beta}^{2\beta-2}(2S_2Q_{\beta}+S_1^2)b^2+\Big[\beta(\beta-1) Q_{\beta}^{2\beta-4}(2S_2Q_{\beta}+S_1)^2\\
&\quad+\beta Q_{\beta}^{2\beta-2}(S_2^2+2Q_{\beta}\Re\pmb{R}_{4,0}+2S_1S_3)\Big]b^4+O(b^6)
\end{align*}
Hence, we have the following equation for $\pmb{R}_{4,0}$,
\begin{equation}
\label{232}
\pmb{L}^{\beta}\pmb{R}_{4,0}=\pmb{J}\Lambda \pmb{R}_{3,0}+\beta Q_{\beta}^{2\beta-2}(2S_2Q_{\beta}+S_1^2)\pmb{R}_{2,0}+H\pmb{Q}_{\beta},
\end{equation}
where $H=\beta^3(\beta-1) Q_{\beta}^{2\beta-4}(2S_2Q_{\beta}+S_1)^2+\beta Q_{\beta}^{2\beta-2}(S_2^2+2S_1S_3)$ is a real valued even function.
The solvability of \eqref{232} reduces to
\begin{equation}
\label{233}
\Big(\nabla Q_{\beta}, -\Lambda S_3+\beta Q_{\beta}^{2\beta-2}(2S_2Q_{\beta}+S_1^2)S_2+HQ_{\beta}\Big)=0.
\end{equation}
Since $S_1$, $S_2$, $S_3$ and $Q_{\beta}$ are all real valued even functions, the above condition is automatically satisfied. So, there exists a unique $\pmb{R}_{4,0}$ satisfying \eqref{232} with 
\begin{equation}
\label{234}
\pmb{R}_{4,0}=\begin{bmatrix}
S_4\\
0
\end{bmatrix},
\end{equation}
where $S_4=(L_+^\beta)^{-1}\big(-\Lambda S_3+\beta Q_{\beta}^{2\beta-2}(2S_2Q_{\beta}+S_1^2)S_2+HQ_{\beta}\big)$ is an even function.

\textit{Order $\mathcal{O}(b^2v)$}: Recall that $\Re \pmb{R}_{1,0}=\Im \pmb{R}_{2,0}=\Re\pmb{R}_{0,1}=\Im\pmb{R}_{1,1}=0$. We then have the following equation for $\pmb{R}_{2,1}$:
\begin{equation}
\label{235}
\pmb{L}^{\beta}\pmb{R}_{2,1}=\pmb{J}(\Lambda \pmb{R}_{1,1}-\nabla\pmb{R}_{2,0})+\beta Q^{2\beta-2}_{\beta}\big[(2Q_{\beta}S_2+S_1^2)\pmb{R}_{0,1}+(2Q_{\beta}F_1+2S_1G_1)\pmb{R}_{1,0}\big].
\end{equation}
Using the fact that $\Re \pmb{R}_{1,0}=\Im \pmb{R}_{2,0}=\Re\pmb{R}_{0,1}=\Im\pmb{R}_{1,1}=0$ again, we know that the solvability of \eqref{235} is equivalent to
\begin{equation}
\label{236}
\Big(Q_{\beta},-\Lambda F_1+\nabla S_2+\beta Q^{2\beta-2}_{\beta}\big[(2Q_{\beta}S_2+S_1^2)G_1+(2Q_{\beta}F_1+2S_1G_1)S_1\big]\Big)=0
\end{equation}
Since $Q_{\beta}$, $S_1$, $S_2$ are even functions while $F_1$, $G_1$ are odd functions, the above condition is automatically satisfied. So there exists a unique $\pmb{R}_{2,1}$ satisfying \eqref{232} with 
\begin{equation}
\label{237}
\pmb{R}_{2,1}=\begin{bmatrix}
0\\
F_2
\end{bmatrix},
\end{equation}
where $F_2=(L_-^{\beta})^{-1}\big(-\Lambda F_1+\nabla S_2+\beta Q^{2\beta-2}_{\beta}[(2Q_{\beta}S_2+S_1^2)G_1+(2Q_{\beta}F_1+2S_1G_1)S_1]\big)$ is an odd function.

\subsubsection*{\bf Step 2: Regularity and decay estimates for $\pmb{R}_{k,\ell}$.} Let $\mathcal{Y}$ be the set of all smooth functions $f$ satisfying the following property: $\forall k\in\mathbb{N}$, $\exists C_{k}>0$ such that $\forall y\in\mathbb{R}$
$$|\partial_y^k f(y)|\leq C_k(1+|y|)^{-1-\beta-k}.$$
We can easily show that $\mathcal{Y}$ has the following properties
\begin{enumerate}
\item If $f\in\mathcal{Y}$, then $\Lambda f,\nabla f\in \mathcal{Y}$.
\item if $f,g\in\mathcal{Y}$, then $fg\in\mathcal{Y}$.
\item $Q_{\beta}\in \mathcal{Y}$.
\end{enumerate}
Here the first two properties follow from direct computation, and the third one is proved in \cite{FL}.

Combining these properties with the equation of each $\pmb{R}_{k,\ell}$, we have for all $k,\ell$
\begin{equation}\label{238}
\pmb{L}^{\beta}\pmb{R}_{k,\ell}\in\mathcal{Y}.
\end{equation}
Since $[L^\beta_+,\nabla]=(2\beta+1)\nabla(Q_\beta^{2\beta})$ and $[L^\beta_-,\nabla]=\nabla(Q_\beta^{2\beta})$. From \eqref{28}, \eqref{210}, \eqref{211}, we obtain that $\pmb{R}_{k,\ell}\in\mathcal{Y}$, which implies \eqref{24} immediately.

\subsubsection*{\bf Step 3: Error term estimate.} As we mentioned before the regularity and decay condition of $\pmb{R}_{k,\ell}$ implies that for $|b|,|v|$ small enough, $\pmb{\Psi}_{b,v}$ is at least $C^{10}$-differentiable with respect to $b$ and $v$. Hence, by Taylor's expansion formula and the construction of $\pmb{R}_{k,\ell}$, we have
\begin{equation}
\label{239}
\pmb{\Psi}_{b,v}=b^5\pmb{T}_{5,0}+vb^3\pmb{T}_{3,1}+v^2b\pmb{T}_{1,2}+v^3\pmb{T}_{0,3},
\end{equation}
where
\begin{align*}
&\pmb{T}_{5,0}=\frac{1}{5!}\int_0^1\partial_b^5\pmb{\Psi}_{tb,0}\,\text{d}t,\quad\pmb{T}_{3,1}=\frac{1}{3!}\int_0^1\partial_b^3\partial_v\pmb{\Psi}_{tb,0}\,\text{d}t,\\
&\pmb{T}_{1,2}=\frac{1}{2!}\int_0^1\partial_b\partial_v^2\pmb{\Psi}_{tb,0}\,\text{d}t,\quad\pmb{T}_{0,3}=\frac{1}{3!}\int_0^1\partial_v^3\pmb{\Psi}_{b,tv}\,\text{d}t.
\end{align*}
Combining the regularity and decay estimate \eqref{24} for $\pmb{R}_{k,\ell}$, we obtain \eqref{23} and \eqref{219}.

\subsubsection*{\bf Step 4: Energy estimate.} The energy estimate \eqref{240} follows from direct computation of the equation of $\pmb{W}_{b,v}$. First, by integrating by parts, one may easily prove the following identity:
\begin{equation}
\label{241}
\beta E(\pmb{W}_{b,v})=\Re\int\overline{\big(|D|^{\beta}\pmb{W}_{b,v}+\pmb{W}_{b,v}-|\pmb{W}_{b,v}|^{2\beta}\pmb{W}_{b,v}-\pmb{J}b\Lambda\pmb{W}_{b,v}\big)}\Lambda\pmb{W}_{b,v}.
\end{equation}
Combining \eqref{22} and \eqref{241}, we have
\begin{align}
\label{242}
\beta E(\pmb{W}_{b,v})=&\Re\int\overline{\pmb{\Psi}_{b,v}}(\Lambda\pmb{W}_{b,v})-v\Im\int\nabla\overline{\pmb{W}_{b,v}}(\Lambda\pmb{W}_{b,v})\nonumber\\
&-bv\Im\int\partial_v(\overline{\pmb{W}_{b,v}})\Lambda\pmb{W}_{b,v}:=I+II+III.
\end{align}
For $I$, from the estimate \eqref{23} we obtain
\begin{equation}
\label{243}
|I|\lesssim |b|^5+v^2(|v|+|b|).
\end{equation}
For $II$ and $III$, from the construction of $\pmb{W}_{b,v}$, we know that $\pmb{W}_{b,v}$ has the following form
\begin{align*}
\pmb{W}_{b,v}=&\begin{bmatrix}
\mathrm{even}\\
0
\end{bmatrix}
+b\begin{bmatrix}
0\\
\mathrm{even}
\end{bmatrix}
+b^2\begin{bmatrix}
\mathrm{even}\\
0
\end{bmatrix}
+b^3\begin{bmatrix}
0\\
\mathrm{even}
\end{bmatrix}
+b^4\begin{bmatrix}
\mathrm{even}\\
0
\end{bmatrix}\\
&+v\begin{bmatrix}
0\\
\mathrm{odd}
\end{bmatrix}
+v^2\begin{bmatrix}
\mathrm{even}\\
0
\end{bmatrix}
+bv\begin{bmatrix}
\mathrm{odd}\\
0
\end{bmatrix}
+vb^2\begin{bmatrix}
0\\
\mathrm{odd}
\end{bmatrix}.
\end{align*}
Hence, we have
\begin{align}
\label{244}
II=-v^2\bigg(\int\nabla Q_\beta\Lambda G_1-\nabla G_1\Lambda Q_\beta\bigg)+O(|b|+|v|)v^2.
\end{align}
Using the commutator formula $[\nabla,\Lambda]=\nabla$ and \eqref{27}, we have
$$\int\nabla Q_\beta\Lambda G_1-\nabla G_1\Lambda Q_\beta=\big(G_1,[\nabla,\Lambda]Q_\beta\big)=(G_1,\nabla Q_\beta)=-(L_-^\beta G_1,G_1)<0.$$
Similarly, we have
\begin{equation}
\label{245}
III=-bv\Im\int\partial_v(\overline{\pmb{W}_{b,v}})\Lambda\pmb{W}_{b,v}=O(|b|+|v|)v^2.
\end{equation}
Injecting \eqref{243}--\eqref{245} into \eqref{242}, we obtain \eqref{240} immediately.

\subsubsection*{\bf Step 5: Scaling invariance.} Finally, for the scaling invariance, we have for any regular enough functions $\omega$ and $\Omega$ satisfying
$$|D|^{\beta}\omega+\omega-|\omega|^{2\beta}\omega=\Omega,$$
we consider $\omega_{\lambda}(y)=\lambda^{1/2}\omega(\lambda y)$, $\Omega_{\lambda}(y)=\lambda^{1/2}\Omega(\lambda y)$ for $\lambda>0$. Then we have
\begin{equation}
\label{247}
|D|^{\beta}\omega_{\lambda}+\lambda^\beta\omega_\lambda-|\omega_\lambda|^{2\beta}\omega_\lambda=\lambda^\beta\Omega_\lambda.
\end{equation}
Differentiating \eqref{247} with respect to $\lambda$ and taking $\lambda=1$, we obtain 
\begin{equation}\label{248}
|D^{\beta}|(\Lambda\omega)+\Lambda\omega-\Lambda\omega|\omega|^{2\beta}-2\beta\omega|\omega|^{2\beta-2}\Re(\omega\Lambda\overline{\omega})=\beta(\Omega-\omega)+\Lambda\Omega,
\end{equation}
where we use the following facts for the above estimate:
$$\bigg(\frac{d}{d\lambda}\omega_\lambda\bigg)\bigg|_{\lambda=1}=\Lambda\omega,\quad\bigg(\frac{d}{d\lambda}|\omega_\lambda|^{2\beta}\bigg)\bigg|_{\lambda=1}=2\beta|\omega|^{2\beta-2}\Re(\omega\Lambda\overline{\omega}).$$
Now, we apply \eqref{248} to $\omega=\pmb{W}_{b,v}$ and 
$$\Omega=\pmb{\Psi}_{b,v}+\pmb{J}b\Lambda\pmb{W}_{b,v}-\pmb{J}v\nabla\pmb{W}_{b,v}-\pmb{J}bv\partial_v\pmb{W}_{b,v},$$
then we obtain \eqref{246}, which concludes the proof of Proposition \ref{P1}.
\end{proof}

\section{Modulation theory}
In this section, we build a general setting of solution with negative energy and slightly super critical mass to \eqref{CP}. We use the variation properties of the ground state $Q_{\beta }$ and conservation laws of \eqref{CP} to establish a sharp geometrical decomposition for such solution and study its basic properties.
\subsection{Geometrical decomposition}
We start with the variation structure of $Q_{\beta}$.
\begin{lemma}[Variation characterization]\label{L3}
There exists $\alpha_1>0$ such that for all $0<\alpha'<\alpha_1$ and $u_0\in H^{\frac{\beta}{2}}(\mathbb{R})$ satisfying
\begin{equation}
\label{31}
\bigg|\int|u_0|^2-\int Q^2_{\beta}\bigg|\leq\alpha',\quad E(u_0)\leq \alpha'\int\big||D|^{\frac{\beta}{2}}u_0\big|^2,
\end{equation}
there exist constants $\lambda_0>0$, $x_0\in\mathbb{R}$, $\gamma_0\in\mathbb{R}$ such that
\begin{equation}\label{32}
\|Q_{\beta}(\cdot)-e^{i\gamma_0}\lambda_0^{1/2}u_0(\lambda_0\cdot+x_0)\|_{H^{\frac{\beta}{2}}}\leq\delta(\alpha').
\end{equation}
\end{lemma} 
\begin{proof}
The proof is based on the variational properties of the ground state and a standard concentration compactness argument. We refer to \cite[Lemma 9]{KMR1}, \cite[Lemma 1]{M1}, \cite[Lemma 1]{MR1} and the references therein for detailed proof.
\end{proof}

We now turn to the Cauchy problem \eqref{CP}. Let $u(t)\in\mathcal{C}([0,T),H^{\frac{\beta}{2}})$ be a solution of \eqref{CP} with maximal lifespan $T>0$. Suppose the initial data $u_0\in\mathcal{B}_{\alpha_0}$ with $E(u_0)<0$ for some small enough $\alpha_0>0$. Then from the mass conservation law \eqref{11} and energy conservation law \eqref{12}, we have for all $0\leq t<T$, there exists $\lambda_1(t)>0$, $x_1(t)\in \mathbb{R}$, $\gamma_1(t)\in\mathbb{R}$ such that
\begin{equation}\label{33}
\Big\|Q_{\beta}(\cdot)-e^{-i\gamma_1(t)}[\lambda_1(t)]^{1/2}u\big(t,\lambda_1(t)\cdot+x_1(t)\big)\Big\|_{H^{\frac{\beta}{2}}}\leq\delta(\alpha_0),
\end{equation}
provided that $\alpha_0>0$ is small enough.

Now we can introduce the \emph{geometrical decomposition} for solutions to \eqref{CP} with negative energy and slightly supercritical mass.

\begin{proposition}[Geometrical decomposition]\label{P2}
Let $u(t)$ be a solution to \eqref{CP} satisfying the conditions mentioned in Theorem \ref{MT1}. Then there exist five $C^1$ functions on $[0,T)$: $\lambda(t)$, $x(t)$, $b(t)$, $v(t)$, $\gamma(t)$ such that the following holds true:
\begin{enumerate}
\item Geometrical decomposition: for all $t\in[0,T)$
\begin{equation}
\label{GD}
u(t,x)=\frac{1}{\lambda^{\frac{1}{2}}(t)}\big[W_{b(t),v(t)}+\varepsilon(t)\big]\bigg(\frac{x-x(t)}{\lambda(t)}\bigg)e^{i\gamma(t)},
\end{equation}
where $W_{b,v}$ is the nonlinear blow-up profile%
\footnote{Here we use the usual notation of complex valued functions instead of the vector form introduced in Section \ref{S1}.}
 constructed in Proposition \ref{P1}. 
\item Orthogonality conditions: for all $t\in[0,T)$
\begin{gather}
(\varepsilon_1,\Lambda \Theta)-(\varepsilon_2, \Lambda \Sigma)=0,\label{34}\\
(\varepsilon_1,\partial_b \Theta)-(\varepsilon_2,\partial_{b}\Sigma)=0,\label{35}\\
(\varepsilon_1,\partial_v \Theta)-(\varepsilon_2,\partial_{v}\Sigma)=0,\label{36}\\
(\varepsilon_1,\nabla \Theta)-(\varepsilon_2,\nabla\Sigma)=0,\label{37}\\
(\varepsilon_1,\Lambda^2 \Theta)-(\varepsilon_2, \Lambda^2 \Sigma)=0.\label{38}
\end{gather}
where
\begin{equation}
\label{39}
W_{b,v}=\Sigma+i\Theta,\quad \varepsilon=\varepsilon_1+i\varepsilon_2.
\end{equation} 
\item A priori estimates on the parameters: for all $t\in[0,T)$
\begin{equation}
\label{310}
|b(t)|+|v(t)|+\|\varepsilon(t)\|_{H^{\frac{\beta}{2}}}\leq \delta(\alpha_0).
\end{equation}
\end{enumerate}
\end{proposition}
\begin{proof}
The proof of Proposition \ref{P2} follows from a standard argument of implicit function theory. We refer to  \cite[Appendix C]{KLR} and \cite[Lemma 2]{MR1} for detailed proof.
\end{proof}
 
\subsection{Modulation estimates} 
With the geometrical decomposition obtained in Proposition \ref{P2}, we are now able to derive some crucial estimates for the parameters $(\lambda(t)$, $b(t)$, $x(t)),v(t),\gamma(t))$. 

We first introduce the following scaling invariant coordinate:
\begin{equation}
\label{311}
s=\int_0^t\frac{1}{\lambda^{\beta}(\tau)}\,\textrm{d}\tau,\quad y=\frac{x-x(t)}{\lambda(t)}.
\end{equation}
\begin{remark}
We mention here that the maximal lifespan of the solution in the rescaled setting is always $+\infty$:
$$s(T)=+\infty.$$
This is a consequence of the scaling structure of the equation. If the solution does not blow up in finite time, then it is obviously that $s(T)=+\infty$, since the negative energy condition removes the possibility that $\lambda(s)\rightarrow+\infty$. If  the solution blows up in finite time $T<+\infty$, then the scaling structure ensures that $\lambda(t)\lesssim (T-t)^{1/\beta}$, which also implies that $s(T)=+\infty$.
\end{remark}

Under this new coordinate, we have the following \textit{a priori} estimates for the parameters $(\lambda,b,v,x,\gamma)$ and the error term $\varepsilon$:
\begin{proposition}
\label{P3}
For all $s\in[0,+\infty)$, the following estimates hold true:
\begin{enumerate}
\item Equation of $\varepsilon$:
\begin{align}
&b_s(\partial_b\Sigma)+(v_s+bv)\partial_v\Sigma+\partial_s\varepsilon_1-M_-(\varepsilon)+b\Lambda\varepsilon_1-v\cdot\nabla\varepsilon_1\nonumber\\
&=\bigg(\frac{\lambda_s}{\lambda}+b\bigg)(\Lambda\Sigma+\Lambda\varepsilon_1)+\bigg(\frac{x_s}{\lambda}-v\bigg)\cdot(\nabla\Sigma+\nabla\varepsilon_1)\nonumber\\
&\quad+\tilde{\gamma}_s(\Theta+\varepsilon_2)+\Im(\Psi_{b,v})-R_2(\varepsilon),\label{312}\\
&b_s(\partial_b\Theta)+(v_s+bv)\partial_v\Theta+\partial_s\varepsilon_2+M_+(\varepsilon)+b\Lambda\varepsilon_2-v\cdot\nabla\varepsilon_2\nonumber\\
&=\bigg(\frac{\lambda_s}{\lambda}+b\bigg)(\Lambda\Theta+\Lambda\varepsilon_2)+\bigg(\frac{x_s}{\lambda}-v\bigg)\cdot(\nabla\Theta+\nabla\varepsilon_2)\nonumber\\
&\quad-\tilde{\gamma}_s(\Sigma+\varepsilon_1)-\Re(\Psi_{b,v})+R_1(\varepsilon).\label{313}
\end{align}
Here $\tilde{\gamma}(s)=-s-\gamma(s)$ and $M=(M_+,M_-)$ are small perturbation of the linearized operator $L^{\beta}=(L_+^\beta,L_-^\beta)$ given by
\begin{align}
M_+(\varepsilon)=&|D|^{\frac{\beta}{2}}\varepsilon_1+\varepsilon_1-|W_{b,v}|^{2\beta}\varepsilon_1-2\beta|W_{b,v}|^{2(\beta-1)}(\Sigma^2\varepsilon_1+\Sigma\Theta\varepsilon_2),\label{314}\\
M_-(\varepsilon)=&|D|^{\frac{\beta}{2}}\varepsilon_2+\varepsilon_2-|W_{b,v}|^{2\beta}\varepsilon_2-2\beta|W_{b,v}|^{2(\beta-1)}(\Theta^2\varepsilon_2+\Sigma\Theta\varepsilon_1).\label{315}
\end{align}
The nonlinear terms $R_1(\varepsilon)$, $R_2(\varepsilon)$ are given by
\begin{align}
R_1(\varepsilon)=&|W_{b,v}+\varepsilon|^{2\beta}(\Sigma+\varepsilon_1)-\Sigma|W_{b,v}|^{2\beta}\nonumber\\
&-|W_{b,v}|^{2\beta}\varepsilon_1-2\beta|W_{b,v}|^{2(\beta-1)}(\Sigma^2\varepsilon_1+\Sigma\Theta\varepsilon_2),\label{316}\\
R_2(\varepsilon)=&|W_{b,v}+\varepsilon|^{2\beta}(\Theta+\varepsilon_2)-\Theta|W_{b,v}|^{2\beta}\nonumber\\
&-|W_{b,v}|^{2\beta}\varepsilon_2-2\beta|W_{b,v}|^{2(\beta-1)}(\Theta^2\varepsilon_2+\Sigma\Theta\varepsilon_1).\label{317}
\end{align}
\item Estimates induced by the conservation laws:
\begin{align}
\label{318}
&|\lambda^\beta|E_0|-(\varepsilon_1,\Sigma+b\Lambda\Theta-v\nabla\Theta)-(\varepsilon_2,\Theta-b\Lambda\Sigma+v\nabla\Sigma)|\nonumber\\
&\lesssim\bigg(\int\big||D|^{\frac{\beta}{2}}\varepsilon\big|^2+\int|\varepsilon|^2e^{-|y|}\bigg)+|b|^5+v^2.
\end{align}
\item Estimates for the modulation parameters:
\begin{align}
&\bigg|\frac{\lambda_s}{\lambda}+b\bigg|+|v_s+bv|+\bigg|\frac{x_s}{\lambda}-v\bigg|+|b_s|+\bigg|\tilde{\gamma}_s-\frac{1}{\|\Lambda Q_{\beta}\|_{L^2}^2}\big(\varepsilon_1,L^\beta_+(\Lambda^2Q_{\beta})\big)\bigg|\nonumber\\
&\lesssim\delta(\alpha_0)\bigg(\int\big||D|^{\frac{\beta}{2}}\varepsilon\big|^2+\int|\varepsilon|^2e^{-|y|}\bigg)^{\frac{1}{2}}+|b|^5+v^2+\lambda^\beta|E_0|.\label{319}
\end{align}
\end{enumerate}
\end{proposition}
\begin{proof}
{\it Proof of (1):} The equations of \eqref{312} and \eqref{313} follows from direct computation.

{\it Proof of (2):} For the estimate \eqref{318}, we expand the energy conservation law as the following:
\begin{align}
\label{320}
&2(\varepsilon_1,\Sigma+b\Lambda\Theta-v\nabla\Theta)+2(\varepsilon_2,\Theta-b\Lambda\Sigma+v\nabla\Sigma)\nonumber\\
=&2\big(\varepsilon_1,bv\partial_v\Theta+\Re(\Psi_{b,v})\big)-2\big(\varepsilon_2,bv\partial_v\Sigma-\Im(\Psi_{b,v})\big)-2\lambda^\beta E_0+2E(W_{b,v})\nonumber\\
&-\frac{1}{\beta+1}\int F(\varepsilon)+\int\big||D|^{\frac{\beta}{2}}\varepsilon\big|^2-\int\big(|W_{b,v}|^{2\beta}+2\beta\Sigma^2|W_{b,v}|^{2\beta-2}\big)\varepsilon^2_1\nonumber\\
&-\int\big(|W_{b,v}|^{2\beta}+2\beta\Theta^2|W_{b,v}|^{2\beta-2}\big)\varepsilon^2_2-4\beta\int\Sigma\Theta|W_{b,v}|^{2\beta-2}\varepsilon_1\varepsilon_2,
\end{align}
where 
\begin{align}
\label{321}
F(\varepsilon)=&|W_{b,v}+\varepsilon|^{2\beta+2}-|W_{b,v}|^{2\beta+2}-(2\beta+2)|W_{b,v}|^{2\beta}(\Sigma\varepsilon_1+\Theta\varepsilon_2)\nonumber\\
&-(\beta+1)\big(|W_{b,v}|^{2\beta}+2\beta\Sigma^2|W_{b,v}|^{2\beta-2}\big)\varepsilon^2_1\nonumber\\
&-(\beta+1)\big(|W_{b,v}|^{2\beta}+2\beta\Theta^2|W_{b,v}|^{2\beta-2}\big)\varepsilon^2_2\nonumber\\
&-4\beta(\beta+1)\Sigma\Theta|W_{b,v}|^{2\beta-2}\varepsilon_1\varepsilon_2.
\end{align}
Applying the following estimate%
\footnote{We use the fact that $W_{b,v}(y)\not=0$ if $b$, $v$ are small enough.}%
\begin{align}\label{322}
&\Big||1+z|^{2+2\beta}-1-(2+\beta)\Re z-(\beta+1)(2\beta+1)(\Re z)^2-(\beta+1)(\Im z)^2 \Big|\nonumber\\
&\leq C(|z|^3+|z|^{2+2\beta}),\quad \forall z\in\mathbb{C}
\end{align}
to $z=\varepsilon/W_{b,v}$, we have%
\footnote{Here we use the Gagliardo-Nirenberg's inequality:
$$\|\varepsilon\|^p_{L^{p}}\leq C\big\||D|^{\frac{\beta}{2}}\varepsilon\big\|_{L^2}^{\frac{p-2}{\beta}}\|\varepsilon\|_{L^2}^{\frac{p(\beta-1)+2}{\beta}},\quad\forall p\geq2$$
and the decay estimate \eqref{24} for the inequality \eqref{323}.}%
\begin{align}
\label{323}
&\int|F(\varepsilon)|\lesssim \int|\varepsilon|^{2\beta+2}+\int|\varepsilon|^3|W_{b,v}|^{2\beta-1}\nonumber\\
&\lesssim\|\varepsilon\|_{L^2}^{2\beta}\bigg(\int\big||D|^{\frac{\beta}{2}}\varepsilon\big|^2\bigg)+\bigg(\int|\varepsilon|^{(\beta+3)}\bigg)^{\frac{1}{1+\beta}}\bigg(\int|\varepsilon|^2|W_{b,v}|^{\frac{(\beta+1)(2\beta-1)}{\beta}}\bigg)^{\frac{\beta}{\beta+1}}\nonumber\\
&\lesssim\|\varepsilon\|_{H^{\frac{\beta}{2}}}\bigg(\int\big||D|^{\frac{\beta}{2}}\varepsilon\big|^2+\int\frac{|\varepsilon|^2}{(1+|y|)^{\beta}}\bigg),
\end{align}

Similarly, we have 
\begin{align}
\label{324}
&\Bigg|\int\big(|W_{b,v}|^{2\beta}+2\beta\Sigma^2|W_{b,v}|^{2\beta-2}\big)\varepsilon^2_1+\int\big(|W_{b,v}|^{2\beta}+2\beta\Theta^2|W_{b,v}|^{2\beta-2}\big)\varepsilon^2_2\nonumber\\
&-4\beta\int\Sigma\Theta|W_{b,v}|^{2\beta-2}\varepsilon_1\varepsilon_2\Bigg|\lesssim \int\frac{|\varepsilon|^2}{(1+|y|)^{\beta}}
\end{align}
and
\begin{align}
\label{325}
&\big|2\big(\varepsilon_1,bv\partial_v\Theta+\Re(\Psi_{b,v})\big)-2\big(\varepsilon_2,bv\partial_v\Sigma-\Im(\Psi_{b,v})\big)\big|\nonumber\\
&\lesssim|b|^5+v^2+\int\frac{|\varepsilon|^2}{(1+|y|)^{\beta}}.
\end{align}

Injecting \eqref{323}--\eqref{349} into \eqref{321}, using \eqref{240}, the \textit{a priori} smallness estimate \eqref{310} and the energy condition $E_0<0$, we obtain
\begin{align}
\label{326}
&|\lambda^\beta|E_0|-(\varepsilon_1,\Sigma+b\Lambda\Theta-v\nabla\Theta)-(\varepsilon_2,\Theta-b\Lambda\Sigma+v\nabla\Sigma)|\nonumber\\
&\lesssim\bigg(\int\big||D|^{\frac{\beta}{2}}\varepsilon\big|^2+\int\frac{|\varepsilon|^2}{(1+|y|)^{\beta}}\bigg)+|b|^5+v^2,
\end{align}
provided that $\alpha_0$ is small enough. 

Now, it is easy to see that the estimate \eqref{318} follows from the following Hardy's type estimate:
\begin{lemma}[Hardy's type estimate]\label{L2}
Suppose $1\leq\beta <2$, then there exists a universal constant $C$ independent of $\beta$ such that for all $f\in H^{\frac{\beta}{2}}(\mathbb{R})$, we have
\begin{equation}\label{327}
\int\frac{|f|^2}{(1+|y|)^\beta}\leq C\bigg(\int\big||D|^{\frac{\beta}{2}}f\big|^2+\int|f|^2e^{-|y|}\bigg).
\end{equation} 
\end{lemma}
\begin{remark}
The proof of \eqref{327} follows from the fractional Hardy's inequality%
\footnote{See \cite{BD} for more details.}
 in dimension $1$  as well as a localization argument. We leave the detailed proof in Appendix \ref{A4}.
\end{remark}

{\it Proof of (3):} Now, we turn to the proof of \eqref{319}. We first differentiate \eqref{34} to obtain
\begin{align}
\label{328}
&(\partial_s\varepsilon_1,\Lambda\Theta)-(\partial_s\varepsilon_2,\Lambda\Sigma)+b_s\Big[\big(\varepsilon_1,\Lambda(\partial_b\Theta)\big)-\big(\varepsilon_2,\Lambda(\partial_b\Sigma)\big)\Big]\nonumber\\
&=-v_s\Big[\big(\varepsilon_1,\Lambda(\partial_v\Theta)\big)-\big(\varepsilon_2,\Lambda(\partial_v\Sigma)\big)\Big].
\end{align}
Then we project \eqref{312} and \eqref{313} onto $-\Lambda\Theta$ and $\Lambda\Sigma$ using \eqref{328} to obtain
\begin{align}\label{329}
&b_s\Big[\big(\partial_b \Theta,\Lambda \Sigma\big)-\big(\partial_b\Sigma,\Lambda\Theta\big)-\big(\varepsilon_2,\Lambda(\partial_b\Sigma)\big)+\big(\varepsilon_1,\Lambda(\partial_b\Theta)\big)\Big]\nonumber\\
&+(v_s+bv)\Big[\big(\partial_v \Theta,\Lambda \Sigma\big)-\big(\partial_v\Sigma,\Lambda\Theta\big)-\big(\varepsilon_2,\Lambda(\partial_v\Sigma)\big)+\big(\varepsilon_1,\Lambda(\partial_v\Theta)\big)\Big]\nonumber\\
&+\bigg(\frac{x_s}{\lambda}-v\bigg)\Big\{\big(\Theta,(\Lambda\Sigma)_y\big)-\big(\Sigma,(\Lambda\Theta)_y\big)+\big(\varepsilon_2,(\Lambda\Sigma)_y\big)-\big(\varepsilon_1,(\Lambda\Theta)_y\big)\Big\}\nonumber\\
&+\bigg(\frac{\lambda_s}{\lambda}+b\bigg)\big\{-(\Lambda\Theta,\Lambda\Sigma)+(\Lambda\Sigma,\Lambda\Theta)+(\varepsilon_2,\Lambda^2\Sigma)-(\varepsilon_1,\Lambda^2\Theta)\big\}\nonumber\\
&-\tilde{\gamma}_s\big\{(\Sigma,\Lambda\Sigma)+(\Theta,\Lambda\Theta)+(\varepsilon_1,\Lambda\Sigma)+(\varepsilon_2,\Lambda\Theta)\big\}\nonumber\\
&=-bv\Big\{\big(\varepsilon_2,\Lambda(\partial_v\Sigma)\big)-\big(\varepsilon_1,\Lambda(\partial_v\Theta)\big)\Big\}-\big(M_+(\varepsilon),\Lambda\Sigma\big)-\big(M_-(\varepsilon),\Lambda\Theta\big)\nonumber\\
&\quad-b\big\{(\Lambda\varepsilon_2,\Lambda\Sigma)-(\Lambda\varepsilon_1,\Lambda\Theta)\big\}+v\big\{(\nabla\varepsilon_2,\Lambda\Sigma)-(\nabla\varepsilon_1,\Lambda\Theta)\big\}\nonumber\\
&\quad-\big(\Lambda \Sigma,\Re(\Psi_{b,v})\big)-\big(\Lambda\Theta,\Im(\Psi_{b,v})\big)+\big(R_1(\varepsilon),\Lambda\Sigma\big)+(R_{2}(\varepsilon),\Lambda\Theta).
\end{align}
By applying the same argument to \eqref{35}--\eqref{38}, we can obtain the other four equations similar as \eqref{329}. For simplicity, we do not write down the other four equations explicitly. 

Now, we view these five equations as a linear system of $$\bigg(b_s,\frac{\lambda_s}{\lambda}+b,\frac{x_s}{\lambda}-v,v_s+bv,\tilde{\gamma}_s\bigg).$$
Hence, \eqref{329} and the other four similar equations can  be written in the following form
\begin{equation}
\label{330}
\begin{bmatrix}
b_s,\frac{\lambda_s}{\lambda}+b,\frac{x_s}{\lambda}-v,v_s+bv,\tilde{\gamma}_s
\end{bmatrix}A
=\begin{bmatrix}
B_1,B_2,B_3,B_4,B_5
\end{bmatrix},
\end{equation}
where $A=A(b,v,\varepsilon)$ is a $5\times5$ matrix and 
\begin{align}
B_1&=-bv\Big\{\big(\varepsilon_2,\partial_v(\Lambda\Sigma)\big)-\big(\varepsilon_1,\partial_v(\Lambda\Theta)\big)\Big\}-\big(M_+(\varepsilon),\Lambda\Sigma\big)-\big(M_-(\varepsilon),\Lambda\Theta\big)\nonumber\\
&\quad-b\big\{(\Lambda\varepsilon_2,\Lambda\Sigma)-(\Lambda\varepsilon_1,\Lambda\Theta)\big\}+v\big\{(\nabla\varepsilon_2,\Lambda\Sigma)-(\nabla\varepsilon_1,\Lambda\Theta)\big\}\nonumber\\
&\quad-\big(\Lambda \Sigma,\Re(\Psi_{b,v})\big)-\big(\Lambda\Theta,\Im(\Psi_{b,v})\big)+\big(R_1(\varepsilon),\Lambda\Sigma\big)+(R_{2}(\varepsilon),\Lambda\Theta).\label{331}
\end{align}
While for $B_i$, $i=2,3,4,5$, they are defined similarly as $B_1$ with $\Lambda\Sigma$ replaced by $\partial_b\Sigma$, $\partial_v\Sigma$, $\nabla\Sigma$, $\Lambda^2\Sigma$ and $\Lambda \Theta$ replaced by $\partial_b\Theta$, $\partial_v\Theta$, $\nabla\Theta$, $\Lambda^2\Theta$ respectively. 

We claim that
\begin{align}
&|B_1|+|B_2|+|B_3|+|B_4|+\big|B_5-\big(\varepsilon_1,L_+^\beta(\Lambda^2Q_\beta)\big)\big|\nonumber\\
&\lesssim\delta(\alpha_0)\bigg(\int\big||D|^{\frac{\beta}{2}}\varepsilon\big|^2+\int|\varepsilon|^2e^{-|y|}\bigg)^{\frac{1}{2}}+|b|^5+v^2+\lambda^\beta|E_0|.\label{335}
\end{align}
Indeed, from Proposition \ref{P1}, the \textit{a priori} estimate \eqref{310} and the Hardy's type estimate \eqref{327}, we can easily obtain that
\begin{align}
\label{332}
&\big|B_1+\big(M_+(\varepsilon),\Lambda\Sigma\big)+\big(M_-(\varepsilon),\Lambda\Theta\big)-\big(R_1(\varepsilon),\Lambda\Sigma\big)-\big(R_{2}(\varepsilon),\Lambda\Theta\big)\big|\nonumber\\
&\lesssim\delta(\alpha_0)\bigg(\int\big||D|^{\frac{\beta}{2}}\varepsilon\big|^2+\int|\varepsilon|^2e^{-|y|}\bigg)^{\frac{1}{2}}+|b|^5+v^2,
\end{align}
Then we apply the following inequality:
$$\forall z\in\mathbb{C},\quad\big|(1+z)|1+z|^{2\beta}-1-(2\beta+1)\Re z-i\Im z\big|\lesssim(|z|^2+|z|^{2\beta+1})$$
to
 $z=\frac{\varepsilon}{W_{b,v}}$, using similar argument as \eqref{323} to obtain
\begin{align}
\label{337}
&\big|\big(R_1(\varepsilon),\Lambda\Sigma\big)+\big(R_{2}(\varepsilon),\Lambda\Theta\big)\big|\lesssim\int(|\varepsilon|^{2\beta+1}+|\varepsilon|^2)|\Lambda W_{b,v}|\nonumber\\
&\lesssim\int|\varepsilon|^{2\beta+2}+\int\frac{|\varepsilon|^2}{(1+|y|)^\beta}\lesssim\int\big||D|^{\frac{\beta}{2}}\varepsilon\big|^2+\int|\varepsilon|^2e^{-|y|}.
\end{align}
Using the relation $L_+^\beta\Lambda Q_\beta=-\beta Q_\beta$, we also have
\begin{align}
\label{338}
(M_+(\varepsilon),\Lambda\Sigma)+(M_-(\varepsilon),\Lambda\Theta)&=-\beta(\varepsilon_1,Q_\beta)+O\bigg(\delta(\alpha_0)\bigg(\int|\varepsilon|^2e^{-|y|}\bigg)^{\frac{1}{2}}\bigg).
\end{align}
Combining \eqref{318}, \eqref{332}, \eqref{337} and \eqref{338}, we obtain that
\begin{equation}
\label{340}
|B_1|\lesssim \delta(\alpha_0)\bigg(\int\big||D|^{\frac{\beta}{2}}\varepsilon\big|^2+\int|\varepsilon|^2e^{-|y|}\bigg)^{\frac{1}{2}}+|b|^5+v^2+\lambda^\beta|E_0|.
\end{equation}
Similarly, we have
\begin{align}
&|B_2+\big(M_+(\varepsilon),\partial_b\Sigma\big)+\big(M_-(\varepsilon),\partial_b\Theta\big)|+|B_3+\big(M_+(\varepsilon),\partial_v\Sigma\big)+\big(M_-(\varepsilon),\partial_v\Theta\big)|\nonumber\\
&+|B_4+\big(M_+(\varepsilon),\nabla\Sigma\big)+\big(M_-(\varepsilon),\nabla\Theta\big)|+|B_5+\big(M_+(\varepsilon),\Lambda^2\Sigma\big)+\big(M_-(\varepsilon),\Lambda^2\Theta\big)|\nonumber\\
&\lesssim\delta(\alpha_0)\bigg(\int\big||D|^{\frac{\beta}{2}}\varepsilon\big|^2+\int|\varepsilon|^2e^{-|y|}\bigg)^{\frac{1}{2}}+|b|^5+v^2.\label{341}
\end{align}

From the orthogonality condition \eqref{34} and \eqref{37} as well as the translation invariance \eqref{26}, we have:
\begin{align}
&(L_-^\beta\varepsilon_2,S_1)=(\varepsilon_2,\Lambda Q_\beta)=O\bigg(\delta(\alpha_0)\bigg(\int\frac{|\varepsilon|^2}{(1+|y|)^\beta}\bigg)^{\frac{1}{2}}\bigg),\label{348}\\
&(L_-^\beta\varepsilon_2,G_1)=(\varepsilon_2,-\nabla Q_\beta)=O\bigg(\delta(\alpha_0)\bigg(\int\frac{|\varepsilon|^2}{(1+|y|)^\beta}\bigg)^{\frac{1}{2}}\bigg),\label{345}\\
&(L_+^\beta\varepsilon_1,\nabla Q_\beta)=0\label{349}.
\end{align}
Using the same argument as \eqref{338}, we obtain the estimates for $B_i$, $i=2,3,4,5$, together with \eqref{340}, we conclude the proof of \eqref{335}.

Finally, for the matrix $A$, by direct computation, we have
\begin{align}
\label{333}
&A(0,0,0)=\nonumber\\
&\begin{bmatrix}
(-S_1,\Lambda Q_\beta)&0&0&0&(\Lambda^2Q_\beta,S_1)\\
0&(-S_1,\Lambda Q_\beta)&0&0&0\\
0&0&(G_1,\nabla Q_\beta)&0&0\\
0&0&0&(G_1,\nabla Q_\beta)&0\\
0&0&0&0&(Q_{\beta},\Lambda^2Q_\beta)
\end{bmatrix}.
\end{align}
Since $A(b,v,\varepsilon)=A(0,0,0)+O(|b|,|v|,\|\varepsilon\|_{H^{\frac{\beta}{2}}})$, 
Combining with \eqref{333}, we have
\begin{equation}
\label{334}
\det{A}\sim 1.
\end{equation}
Injecting \eqref{333} and \eqref{334} into \eqref{330}, using \eqref{335}, we obtain \eqref{319} immediately.
\end{proof}

\section{The local virial argument}\label{S4}
This section is devoted to exhibit the dispersion structure of the Cauchy problem \eqref{CP} by using the virial identity. 

Similar argument has been introduced by Merle and Rapha\"el \cite{MR3,MR2,MR1,MR5,MR4} for the local case when $\beta=2$, where they exhibit a local dispersive relation which can be roughly written as the following:
\begin{equation}
\label{41}
(\varepsilon_2,Q)_s\geq \bar{H}(\varepsilon,\varepsilon)+\lambda^2|E_0|-Ce^{-\frac{C}{|(\varepsilon_2,Q)|}}.
\end{equation}
Here $\bar{H}(\varepsilon,\varepsilon)$ is some $H^1$ quadratic form of $\varepsilon$, which is positive except on a three dimensional vector space. However, all these three negative direction can be controlled by the conservation laws and  modulation theory up to an exponentially small correction in $(\varepsilon_2,Q)$. Hence, the estimate \eqref{41} implies:
\begin{equation}
\label{42}
(\varepsilon_2,Q)_s\geq \delta_0\bigg(\int|\varepsilon_y|^2+\int|\varepsilon|^2e^{-|y|}\bigg)-Ce^{-\frac{C}{|(\varepsilon_2,Q)|}},
\end{equation}
which will lead to the log-log blow-up dynamics by a standard ODE argument on the parameter..

While for \eqref{CP} in the nonlocal case when $\beta<2$ the following virial identity
\begin{equation}
\label{43}
\frac{d}{dt}\bigg(\Im\int x\cdot \nabla u(t)\bar{u}(t)\bigg)=2E(u_0)
\end{equation}
still holds true, we can generalize \eqref{42} to the fractional case.

\subsection{The virial estimate in $\varepsilon$ variable}In this section, we will derive the virial identity in $\varepsilon$ variable. More precisely, we have the following local virial estimate:

\begin{proposition}[Local virial estimate]\label{P4}
There exists $\beta_0<2$ such that if $\beta_0<\beta<2$ and $0<\alpha_0\leq \alpha^*(\beta)\ll1$,  then there exist universal constants $C>0$, $\mu_0>0$ such that for all $s\in[0,+\infty)$, there holds%
\footnote{We mention here that the parameter $b$ plays a similar role as $(\varepsilon_2,Q)$ in \eqref{42}, since it is governed by the orthogonality condition $(\varepsilon_1,\Theta)-(\varepsilon_2,\Sigma)=0.$}
\begin{equation}
\label{422}
b_s\geq \mu_0\bigg(\lambda^\beta|E_0|+v^2+\int\big||D|^{\frac{\beta}{2}}\varepsilon\big|^2+\int|\varepsilon|^2e^{-|y|}\bigg)-Cb^{10}.
\end{equation}
Moreover, for all $s_1<s_2$, there holds
\begin{equation}
\label{423}
\int_{s_1}^{s_2}\bigg(\lambda^\beta|E_0|+v^2+\int\big||D|^{\frac{\beta}{2}}\varepsilon\big|^2+\int|\varepsilon|^2e^{-|y|}\bigg)\lesssim \delta(\alpha_0)+\int_{s_1}^{s_2}b^{10}.
\end{equation}
\end{proposition}

\begin{remark}
Here the error term $|b|^{10}$ comes from the estimate \eqref{23} for the self-similar equation, which will determine the upper bound on the blow-up rate introduced in Theorem \ref{MT1}.
\end{remark}

\begin{proof}
We divide the proof into several steps:
\subsubsection*{\bf Step 1:} Virial identity in $\varepsilon$ variable:

We first claim that for all $\mu_0>0$, there exists $C_{\mu_0}>0$ and $\alpha^*=\alpha^*(\mu_0)$, such that if $\alpha_1<\alpha^*$, there holds
\begin{align}\label{452}
&\bigg|b_s\Big[\big(\partial_b \Theta,\Lambda \Sigma\big)-\big(\partial_b\Sigma,\Lambda\Theta\big)\Big]+v\Big\{\big(\Theta,(\Lambda\Sigma)_y\big)-\big(\Sigma,(\Lambda\Theta)_y\big)\Big\}-\widetilde{H}^\beta(\varepsilon,\varepsilon)+\beta\lambda^\beta E_0\bigg|\nonumber\\
&\leq\mu_1\bigg(\lambda^\beta|E_0|+v^2+\int\big||D|^{\frac{\beta}{2}}\varepsilon\big|^2+\int|\varepsilon|^2e^{-|y|}\bigg)+C_{\mu_1}|b|^{10},
\end{align}
where 
\begin{equation}\label{481}
\widetilde{H}^\beta(\varepsilon,\varepsilon)=\frac{\beta}{2}\big[(\mathcal{L}_1^\beta\varepsilon_1,\varepsilon_1)+(\mathcal{L}_2^\beta\varepsilon_2,\varepsilon_2)\big]-\frac{1}{\|\Lambda Q_\beta\|^2_{L^2}}(\varepsilon_1,L_+^\beta(\Lambda^2Q_\beta))(\varepsilon_1,\Lambda Q_\beta)
\end{equation}
with
$\mathcal{L}_1^{\beta}=|D|^{\beta}+2(2\beta+1)yQ_{\beta}'Q_{\beta}^{2\beta-1}$, $ \mathcal{L}_2^{\beta}=|D|^{\beta}+2 yQ_{\beta}'Q_{\beta}^{2\beta-1}$,
for all $\varepsilon=\varepsilon_1+i\varepsilon_2\in H^{\frac{\beta}{2}}$.

Now, we turn to the proof of \eqref{452}. We calculate every tern in \eqref{329} as follows.

First, from the construction of $W_{b,v}$, we have 
\begin{align}\label{425}
\pmb{W}_{b,v}=&\begin{bmatrix}
\mathrm{even}\\
0
\end{bmatrix}
+b\begin{bmatrix}
0\\
\mathrm{even}
\end{bmatrix}
+b^2\begin{bmatrix}
\mathrm{even}\\
0
\end{bmatrix}
+b^3\begin{bmatrix}
0\\
\mathrm{even}
\end{bmatrix}
+b^4\begin{bmatrix}
\mathrm{even}\\
0
\end{bmatrix}\nonumber\\
&+v\begin{bmatrix}
0\\
\mathrm{odd}
\end{bmatrix}
+v^2\begin{bmatrix}
\mathrm{even}\\
0
\end{bmatrix}
+bv\begin{bmatrix}
\mathrm{odd}\\
0
\end{bmatrix}
+vb^2\begin{bmatrix}
0\\
\mathrm{odd}
\end{bmatrix},
\end{align}
which implies that
\begin{equation}\label{426}
\big|\big(\Theta,(\Lambda\Sigma)_y\big)-\big(\Sigma,(\Lambda\Theta)_y\big)\big|+\big| \big(\partial_v \Theta,\Lambda \Sigma\big)-\big(\partial_v\Sigma,\Lambda\Theta\big)\big|=O(|v|).
\end{equation}
Hence using the modulation estimate \eqref{319}, we have
\begin{align}
\label{427}
&\Big|(v_s+bv)\Big[\big(\partial_v \Theta,\Lambda \Sigma\big)-\big(\partial_v\Sigma,\Lambda\Theta\big)-\big(\varepsilon_2,\Lambda(\partial_v\Sigma)\big)+\big(\varepsilon_1,\Lambda(\partial_v\Theta)\big)\Big]\Big|\nonumber\\
&+\bigg|\bigg(\frac{x_s}{\lambda}-v\bigg)\Big\{\big(\Theta,(\Lambda\Sigma)_y\big)-\big(\Sigma,(\Lambda\Theta)_y\big)+\big(\varepsilon_2,(\Lambda\Sigma)_y\big)-\big(\varepsilon_1,(\Lambda\Theta)_y\big)\Big\}\bigg|\nonumber\\
&+\bigg|\tilde{\gamma}_s\big\{(\varepsilon_1,\Lambda\Sigma)+(\varepsilon_2,\Lambda\Theta)\big\}-\frac{1}{\|\Lambda Q_\beta\|^2_{L^2}}(\varepsilon_1,L_+^\beta\Lambda^2Q_\beta)(\varepsilon_1,\Lambda Q_\beta)\bigg|\nonumber\\
&\lesssim \delta(\alpha_0)\bigg(\lambda^\beta|E_0|+v^2+\int\big||D|^{\frac{\beta}{2}}\varepsilon\big|^2+\int|\varepsilon|^2e^{-|y|}\bigg)+b^{10}
\end{align}
From the orthogonality condition \eqref{38}, we know that
\begin{equation}
\label{424}
\bigg(\frac{\lambda_s}{\lambda}+b\bigg)\big\{(\varepsilon_2,\Lambda^2\Sigma)-(\varepsilon_1,\Lambda^2\Theta)\big\}=0.
\end{equation}

Then,  we use \eqref{246}, the commutator relation $[\nabla,\Lambda]=\nabla$ and the orthogonality condition \eqref{37}  to compute:
\begin{align}
\label{415}
&\big(M_+(\varepsilon),\Lambda\Sigma\big)+\big(M_-(\varepsilon),\Lambda\Theta\big)\nonumber\\
&=\big(\varepsilon_1,|D|(\Lambda\Sigma)+\Lambda\Sigma-|W_{b,v}|^{2\beta}\Lambda\Sigma-2\beta \Sigma|W_{b,v}|^{(2\beta-2)}(\Sigma\Lambda\Sigma+\Theta\Lambda\Theta)\big)\nonumber\\
&\quad+\big(\varepsilon_2,|D|(\Lambda\Theta)+\Lambda\Theta-|W_{b,v}|^{2\beta}\Lambda\Theta-2\beta \Theta|W_{b,v}|^{(2\beta-2)}(\Sigma\Lambda\Sigma+\Theta\Lambda\Theta)\big)\nonumber\\
&=-\beta\Big[\big(\varepsilon_1,\Sigma-b\Lambda\Theta+v\nabla\Theta-\Re(\Psi_{b,v})\big)+(\varepsilon_2,\Theta+b\Lambda\Sigma-v\nabla\Sigma-\Im(\Psi_{b,v}))\Big]\nonumber\\
&\quad+b\big[(\varepsilon_2,\Lambda^2\Sigma)-(\varepsilon_1,\Lambda^2\Theta)\big]-v\big[(\varepsilon_2,\nabla\Lambda\Sigma)-(\varepsilon_1,\nabla\Lambda\Theta)\big]\nonumber\\
&\quad-bv\big[(\varepsilon_2,\Lambda(\partial_v\Sigma))-(\varepsilon_1,\Lambda(\partial_v\Theta))\big]+\big(\varepsilon_1,\Re(\Lambda\Psi_{b,v})\big)+\big(\varepsilon_2,\Im(\Lambda\Psi_{b,v})\big).
\end{align}
Next,  from \eqref{241}, we have
\begin{align}\label{421}
\beta E(W_{b,v})=&(\Re(\Psi_{b,v}),\Lambda\Sigma)+(\Im(\Psi_{b,v}),\Lambda\Theta)-v\Big\{\big(\Theta,(\Lambda\Sigma)_y\big)-\big(\Sigma, (\Lambda\Theta)_y\big)\Big\}\nonumber\\
&-bv\Big\{\big(\Theta,\Lambda(\partial_v\Sigma)\big)-\big(\Sigma,\Lambda(\partial_v\Theta)\big)\Big\}.
\end{align}

Combining \eqref{320}, \eqref{321}, \eqref{427}, \eqref{424}, \eqref{415} and \eqref{421},  we obtain 
\begin{align}
\label{416}
&b_s\Big[\big(\partial_b \Theta,\Lambda \Sigma\big)-\big(\partial_b\Sigma,\Lambda\Theta\big)-\big(\varepsilon_2,\Lambda(\partial_b\Sigma)\big)+\big(\varepsilon_1,\Lambda(\partial_b\Theta)\big)\Big]\nonumber\\
&=-v\Big\{\big(\Theta,(\Lambda\Sigma)_y\big)-\big(\Sigma,(\Lambda\Theta)_y\big)\Big\}-bv\Big\{\big(\Theta,\Lambda(\partial_v\Sigma)\big)-\big(\Sigma,\Lambda(\partial_v\Theta)\big)\Big\}\nonumber\\
&\quad+\beta\lambda^\beta|E_0|-\big(\varepsilon_1,\Re(\Lambda \Psi_{b,v})\big)-\big(\varepsilon_1,\Im(\Lambda \Psi_{b,v})\big)+G(\varepsilon)\nonumber\\
&\quad+O\Bigg(\delta(\alpha_0)\bigg(\lambda^\beta|E_0|+v^2+\int\big||D|^{\frac{\beta}{2}}\varepsilon\big|^2+\int|\varepsilon|^2e^{-|y|}\bigg)+b^{10}\Bigg),
\end{align}
where 
\begin{align}
\label{417}
G(\varepsilon)=&-\frac{\beta}{2+2\beta}\int \bigg(|W_{b,v}+\varepsilon|^{2\beta+2}-|W_{b,v}|^{2\beta+2}-(2\beta+2)|W_{b,v}|^{2\beta}(\Sigma\varepsilon_1+\Theta\varepsilon_2)\bigg)\nonumber\\
&+\frac{\beta}{2}\int\big||D|^{\frac{\beta}{2}}\varepsilon\big|^2+(R_1(\varepsilon),\Lambda\Sigma)+(R_2(\varepsilon),\Lambda\Theta).
\end{align}
By tracking all quadratic term of $\varepsilon$ in $G(\varepsilon)$, using the argument for \eqref{323}, we can easily show that
\begin{align}\label{418}
G(\varepsilon)=\widetilde{H}^\beta(\varepsilon,\varepsilon)+O\Bigg(\delta(\alpha_0)\bigg(\int\big||D|^{\frac{\beta}{2}}\varepsilon\big|^2+\int|\varepsilon|^2e^{-|y|}\bigg)\Bigg).
\end{align}

Finally, from \eqref{425}, we have
\begin{equation}\label{442}
\Big| bv\Big\{\big(\Theta,\Lambda(\partial_v\Sigma)\big)-\big(\Sigma,\Lambda(\partial_v\Theta)\big)\Big\}\Big|\lesssim (|v|+|b|)v^2\lesssim\delta(\alpha_0) v^2.
\end{equation}
From \eqref{319}, we have
\begin{align}\label{453}
&\Big|b_s\big(\varepsilon_2,\Lambda(\partial_b\Sigma)\big)+\big(\varepsilon_1,\Lambda(\partial_b\Theta)\big)\Big|\nonumber\\
&\lesssim \delta(\alpha_0)\bigg(\lambda^\beta|E_0|+v^2+\int\big||D|^{\frac{\beta}{2}}\varepsilon\big|^2+\int|\varepsilon|^2e^{-|y|}\bigg)+b^{10}.
\end{align}
From \eqref{23}, we have:
\begin{align}\label{438}
&\big|\big(\varepsilon_1,\Re(\Lambda \Psi_{b,v})\big)+\big(\varepsilon_2,\Im(\Lambda \Psi_{b,v})\big)\big|\lesssim (b^{5}+(|v|+|b|)v^2)\bigg(\int\big||D|^{\frac{\beta}{2}}\varepsilon\big|^2+\int|\varepsilon|^2e^{-|y|}\bigg)^{\frac{1}{2}}\nonumber\\
&\leq \frac{\mu_1}{1000}\bigg(\int\big||D|^{\frac{\beta}{2}}\varepsilon\big|^2+\int|\varepsilon|^2e^{-|y|}\bigg)+\frac{\mu_1}{1000}v^2+C_{\mu_1}b^{10},
\end{align}
Injecting \eqref{418} --\eqref{438} into \eqref{416}, we obtain \eqref{452} immediately.

\subsubsection*{\bf Step 2:} Spectral property:
Now, we will show that the quadratic form $\widetilde{H}^\beta(\varepsilon,\varepsilon)$ is positive except on a dimension four subspace provided that $\beta$ is close to $2$. We first denote by
$$H^\beta(\varepsilon,\varepsilon)=(\mathcal{L}^\beta_1\varepsilon_1,\varepsilon_1)+(\mathcal{L}_2^\beta\varepsilon_2,\varepsilon_2)$$
Then, we introduce the following \emph{spectral property} for $1\leq\beta\leq2$:

\begin{definition}
We say that \emph{spectral property} holds true for $1\leq\beta\leq2$, if there exists a universal constant $\delta>0$ such that 
\begin{equation}\label{SP}
H^{\beta}(\varepsilon,\varepsilon)\geq \delta \bigg(\int\big||D|^{\frac{\beta}{2}}\varepsilon\big|^2+\int|\varepsilon|^2e^{-|y|}\bigg),
\end{equation}
for all $\varepsilon=\varepsilon_1+i\varepsilon_2\in H^{\frac{\beta}{2}}(\mathbb{R})$ with $(\varepsilon_1,Q_{\beta})=(\varepsilon_1,G_1)=(\varepsilon_2, \Lambda Q_{\beta})=(\varepsilon_2, \Lambda^2 Q_{\beta})=0$.
\end{definition}

We mention here that for general $\beta\in[1,2)$ it is still not known whether the spectral property holds true. But for the local case when $\beta=2$, the spectral property has been proved by Merle and Rapha\"el \cite{MR1} with the help of some numeric tools. As mentioned before the ground state $Q_\beta$ is continuous with respect to $\beta$ up to $\beta=2$. More precisely, we have
\begin{lemma}[Continuity of $Q_\beta$ with respect to $\beta$]\label{L7}
We denote by $Q=Q_2$ for the ground state in the local case when $\beta=2$, and 
$$L_+=-\Delta+1-5Q^4,\quad L_-=-\Delta+1-Q^4$$
the linearized operator at $Q$. Then we have
\begin{enumerate}
\item Continuity of $Q_\beta$:
\begin{equation}
Q_{\beta}\rightarrow Q,\;in\;H^1,\quad as \;\;\beta\rightarrow 2^-.
\end{equation}
\item Uniform boundedness from the above: there exist a constant $C$ independent of $\beta$ such that
\begin{equation}
|Q_\beta(y)|\leq \frac{C}{(1+|y|)^{1+\beta}},\quad |Q(y)|\leq Ce^{-\frac{|y|}{2}}.
\end{equation}
\item Convergence of $L^\beta_{\pm}$ in the norm-resolvent sense:
\begin{equation}
\bigg\|\frac{1}{L^\beta_{\pm}+z}-\frac{1}{L_{\pm}+z}\bigg\|_{L^2\rightarrow L^2}\rightarrow 0,
\end{equation}
as $\beta\rightarrow 2^-$ for all $z\in\mathbb{C}$ with $\Im z\not=0$.
\end{enumerate}
\end{lemma}
Using this property and a perturbation argument, we have:
\begin{proposition}\label{P6}
There exists $\beta_0<2$ such that if $\beta_0<\beta<2$, then the spectral property \eqref{SP} holds true.
\end{proposition}
\begin{proof}
We leave the proof of Proposition \ref{P6} in Appendix \ref{A2}.
\end{proof}

Now, we turn to the estimate of the quadratic form $\widetilde{H}^\beta$.

By assuming the spectral property \eqref{SP} holds true, we claim that there exist universal constants $\delta_0>0$ and $C>0$ such that
\begin{align}\label{433}
\widetilde{H}^\beta(\varepsilon,\varepsilon)\geq \frac{\delta_0}{2}\bigg(\int\big||D|^{\frac{\beta}{2}}\varepsilon\big|^2+\int|\varepsilon|^2e^{-|y|}\bigg)-\delta(\alpha_0)v^2-C\big(b^{10}+(\lambda^\beta|E_0|)^2\big),
\end{align}

To prove \eqref{433}, we first introduce the following elliptic estimate:
\begin{lemma}\label{L5}
Suppose that the spectral property \eqref{SP} holds true for some $\delta>0$. Then there exists some constant $\delta_0\in(0,\delta)$ depending only on $\delta$, such that for all $\tilde{\varepsilon}_1\in H^{\frac{\beta}{2}}(\mathbb{R};\mathbb{R})$,
\begin{align}\label{431}
&\widetilde{H}^\beta(\tilde{\varepsilon}_1,\tilde{\varepsilon}_1)\geq \delta_0\bigg(\int\big||D|^{\frac{\beta}{2}}\tilde{\varepsilon}_1\big|^2+\int|\tilde{\varepsilon}_1|^2e^{-|y|}\bigg)\nonumber\\
&\qquad-\frac{1}{\delta_0} \big[(\tilde{\varepsilon}_1,Q_\beta)^2+(\tilde{\varepsilon}_1,G_1)^2+(\tilde{\varepsilon}_1,S_1)^2\big],
\end{align}
where
$$\widetilde{H}^\beta(\tilde{\varepsilon}_1,\tilde{\varepsilon}_1)=\frac{\beta}{2}(\mathcal{L}_1\tilde{\varepsilon}_1,\tilde{\varepsilon}_1)-\frac{1}{\|\Lambda Q_\beta\|^2_{L^2}}(\tilde{\varepsilon}_1,L_+^\beta\Lambda^2Q_\beta)(\tilde{\varepsilon}_1,\Lambda Q_\beta)$$
and $\delta>0$ is the constant on the right hand side of \eqref{SP}.
\end{lemma}
\begin{proof}[Proof of Lemma \ref{L5}]. We denote by 
\begin{align*}
\widetilde{B}(f,g)=&\frac{\beta}{2}\big[(\mathcal{L}_1f,g)\big]-\frac{1}{2\|\Lambda Q_\beta\|^2_{L^2}}(f,L_+^\beta\Lambda^2Q_\beta)(g,\Lambda Q_\beta)\\
&-\frac{1}{2\|\Lambda Q_\beta\|^2_{L^2}}(g,L_+^\beta\Lambda^2Q_\beta)(f,\Lambda Q_\beta)
\end{align*}
the bilinear form underlying $\widetilde{H}^\beta$ for all $f,g\in H^{\frac{\beta}{2}}(\mathbb{R};\mathbb{R})$.

We claim that for all $\tilde{\varepsilon}_1\in H^{\frac{\beta}{2}}(\mathbb{R};\mathbb{R})$, there holds
\begin{equation}
\label{432}
\widetilde{B}(\tilde{\varepsilon}_1,\Lambda Q_\beta)=0.
\end{equation}
Indeed, by direct computation, we have
\begin{equation}\label{449}
\mathcal{L}_1(\Lambda Q_\beta)=\frac{1}{\beta}\big[L_+^\beta(\Lambda^2Q_\beta)-\Lambda(L_+^\beta\Lambda Q_\beta)\big]=\frac{1}{\beta}L_+^\beta(\Lambda^2Q_\beta)+\Lambda Q_\beta.
\end{equation}
which implies that \eqref{432}.

Next, we let $\tilde{\varepsilon}_1\in H^{\frac{\beta}{2}}(\mathbb{R};\mathbb{R})$ such that $(\tilde{\varepsilon}_1,G_1)=(\tilde{\varepsilon}_1,S_1)=0$. We also set $\hat{\varepsilon}_1=\tilde{\varepsilon}_1+\nu \Lambda Q_\beta$ with $\nu=-(\tilde{\varepsilon}_1,\Lambda Q_\beta)/\|\Lambda Q_\beta\|^2_{L^2}$, so that $(\hat{\varepsilon}_1,G_1)=(\hat{\varepsilon}_1,\Lambda Q_\beta)=0$.

From the Hardy's type estimate \eqref{327}, we have
$$\nu^2=\frac{(\tilde{\varepsilon}_1,\Lambda Q_\beta)^2}{\|\Lambda Q_\beta\|^2_{L^2}}\lesssim \int\frac{|\tilde{\varepsilon}_1|^2}{(1+|y|)^{\beta}}\lesssim \int\big||D|^{\frac{\beta}{2}}\tilde{\varepsilon}_1\big|^2+\int|\tilde{\varepsilon}_1|^2e^{-|y|},$$
which implies that 
\begin{align*}
&\frac{1}{C}\bigg(\int\big||D|^{\frac{\beta}{2}}\tilde{\varepsilon}_1\big|^2+\int|\tilde{\varepsilon}_1|^2e^{-|y|}\bigg)\leq \int\big||D|^{\frac{\beta}{2}}\hat{\varepsilon}_1\big|^2+\int|\hat{\varepsilon}_1|^2e^{-|y|}\\
&\qquad\lesssim C\bigg(\int\big||D|^{\frac{\beta}{2}}\tilde{\varepsilon}_1\big|^2+\int|\tilde{\varepsilon}_1|^2e^{-|y|}\bigg),
\end{align*}
for some universal constant $C>0$.

Then from \eqref{432} and the fact that $(\hat{\varepsilon}_1,\Lambda Q_\beta)=0$, we have
$$\widetilde{H}^\beta(\tilde{\varepsilon}_1,\tilde{\varepsilon}_1)=\widetilde{H}^\beta(\hat{\varepsilon}_1-\nu\Lambda Q_\beta,\hat{\varepsilon}_1-\nu\Lambda Q_\beta)=\widetilde{H}^\beta(\hat{\varepsilon}_1,\hat{\varepsilon}_1)=\frac{\beta}{2}(\mathcal{L}_1\hat{\varepsilon}_1,\hat{\varepsilon}_1).$$

Finally, from the fact that $(Q_\beta,\Lambda Q_\beta)=0$, $(\tilde{\varepsilon}_1,Q_\beta)=(\hat{\varepsilon}_1,Q_\beta)$ and the spectral property \eqref{SP}, we know that there exists some constant $\delta>\delta_0>0$ such that%
\footnote{Here we use the fact that $\beta\geq1$ for the first inequality.}
\begin{align*}
\widetilde{H}^\beta(\tilde{\varepsilon}_1,\tilde{\varepsilon}_1)&=\frac{\beta}{2}(\mathcal{L}_1\hat{\varepsilon}_1,\hat{\varepsilon}_1)\geq \frac{\delta}{2}\bigg(\int\big||D|^{\frac{\beta}{2}}\hat{\varepsilon}_1\big|^2+\int|\hat{\varepsilon}_1|^2e^{-|y|}\bigg)-C_\delta(\hat{\varepsilon}_1,Q_\beta)^2\\
&\geq \delta_0\bigg(\int\big||D|^{\frac{\beta}{2}}\tilde{\varepsilon}_1\big|^2+\int|\tilde{\varepsilon}_1|^2e^{-|y|}\bigg)-\frac{1}{\delta_0}(\tilde{\varepsilon}_1,Q_\beta)^2,
\end{align*}
which implies \eqref{431} immediately. Hence we conclude the proof of Lemma \ref{L5}.
\end{proof}

Now we turn back to the proof of \eqref{433}. From \eqref{431} and the spectral property \eqref{SP}, we know that 
\begin{align}\label{434}
&\frac{\beta}{2}\big[(\mathcal{L}_1\varepsilon_1,\varepsilon_1)+(\mathcal{L}_2\varepsilon_2,\varepsilon_2)\big]-\frac{1}{\|\Lambda Q_\beta\|^2_{L^2}}(\varepsilon_1,L_+^\beta(\Lambda^2Q_\beta))(\varepsilon_1,\Lambda Q_\beta)\nonumber\\
&\geq \delta_0\bigg(\int\big||D|^{\frac{\beta}{2}}\varepsilon\big|^2+\int|\varepsilon|^2e^{-|y|}\bigg)-\frac{1}{\delta_0}\big[(\varepsilon_1,Q_\beta)^2+(\varepsilon_1,S_1)^2+(\varepsilon_1,G_1)^2\nonumber\\
&\qquad+(\varepsilon_2,\Lambda Q_\beta)^2+(\varepsilon_2,\Lambda^2Q_\beta)^2\big].
\end{align}
From the orthogonality condition \eqref{34}--\eqref{38} and the Hardy's type estimate \eqref{327}, we know that%
\footnote{Recall that $S_1$ and $G_1$ are defined in \eqref{215} and \eqref{217}.}
\begin{align}\label{435}
&(\varepsilon_1,S_1)^2+(\varepsilon_1,G_1)^2+(\varepsilon_2,\Lambda Q_\beta)^2+(\varepsilon_2,\Lambda^2Q_\beta)^2\nonumber\\
&\lesssim\delta(\alpha_0)\bigg(\int\big||D|^{\frac{\beta}{2}}\varepsilon\big|^2+\int|\varepsilon|^2e^{-|y|}\bigg).
\end{align}
From \eqref{318}, we can easily obtain
\begin{align}\label{436}
(\varepsilon_1,Q_\beta)^2\lesssim \delta(\alpha_0)\bigg(v^2+\int\big||D|^{\frac{\beta}{2}}\varepsilon\big|^2+\int|\varepsilon|^2e^{-|y|}\bigg)+b^{10}+(\lambda^\beta|E_0|)^2.
\end{align}

Combining \eqref{434}--\eqref{436}, we obtain \eqref{433} immediately.

\subsubsection*{\bf Step 3:} End of the proof:

Recalling from \eqref{425}, we have
\begin{align}
\label{440}
&v\Big\{\big(\Theta,(\Lambda\Sigma)_y\big)-\big(\Sigma,(\Lambda\Theta)_y\big)\Big\}\nonumber\\
&=v^2\big((G_1,\nabla\Lambda Q_\beta)-(G_1,\Lambda\nabla Q_\beta)\big)+O(|b|+|v|)v^2.
\end{align}
From \eqref{27} and the commutator relation $[\nabla,\Lambda]=\nabla$, we have
$$\big((G_1,\nabla\Lambda\Sigma)-(G_1,\Lambda\nabla\Sigma)\big)=(G_1,[\nabla,\Lambda]Q_\beta)=(G_1,\nabla Q_{\beta})=-(L_-^\beta\nabla Q_{\beta},\nabla Q_\beta)<0.$$
Hence, we have
\begin{equation}\label{441}
-v\Big\{\big(\Theta,(\Lambda\Sigma)_y\big)-\big(\Sigma,(\Lambda\Theta)_y\big)\Big\}\geq \frac{c_0}{2}v^2,
\end{equation}
for $c_0= (L_-^\beta\nabla Q_{\beta},\nabla Q_\beta)>0$, provided that $\alpha_0$ is small enough.

Injecting \eqref{433} and \eqref{440} into \eqref{452}, and choosing $\mu_1$ small enough, we have
\begin{align}\label{443}
&b_s\Big[\big(\partial_b \Theta,\Lambda \Sigma\big)-\big(\partial_b\Sigma,\Lambda\Theta\big)\Big]\geq \frac{\delta_0}{2}\bigg(\int\big||D|^{\frac{\beta}{2}}\varepsilon\big|^2+\int|\varepsilon|^2e^{-|y|}\bigg)+\frac{c_0}{2}v^2+\beta\lambda^\beta|E_0|\nonumber\\
&\quad-\frac{\delta_0}{50}\bigg(\int\big||D|^{\frac{\beta}{2}}\varepsilon\big|^2+\int|\varepsilon|^2e^{-|y|}\bigg)-\frac{c_0}{50}v^2-C\big( b^{10}+(\lambda^\beta|E_0|)^2\big)\nonumber\\
&\geq \delta_1\bigg(\lambda^\beta|E_0|+v^2+\int\big||D|^{\frac{\beta}{2}}\varepsilon\big|^2+\int|\varepsilon|^2e^{-|y|}\bigg)-Cb^{10},
\end{align}
for some universal constant $\delta_1>0$, provided that $\alpha_0$ is small. 

Now, we observe from $W_{b,v}|_{b=0,v=0}=Q_\beta$ and $\partial_b W_{b,v}|_{b=0,v=0}=iS_1$ that
\begin{align}\label{444}
&\big(\partial_b \Theta,\Lambda \Sigma\big)-\big(\partial_b\Sigma,\Lambda\Theta\big)-\big(\varepsilon_2,\Lambda(\partial_b\Sigma)\big)+\big(\varepsilon_1,\Lambda(\partial_b\Theta)\big)\nonumber\\
&=(S_1,\Lambda Q_\beta)+O\big(|b|+|v|+\|\varepsilon\|_{H^{\frac{\beta}{2}}}\big)=(L^\beta_-\Lambda Q_\beta,\Lambda Q_\beta)+O(\delta(\alpha_0))>0,
\end{align}
where we use \eqref{27} and \eqref{310} for the last inequality. Injecting \eqref{444} into \eqref{443}, we obtain \eqref{422} immediately. Finally, integrating \eqref{422} from $s_1$ to $s_2$ we obtain \eqref{423}, which concludes the proof of Proposition \ref{P4}.

\end{proof}

\section{Proof of Theorem \ref{MT1}}
In this section, we will finish the proof of Theorem \ref{MT1}. We consider initial data $u_0\in\mathcal{B}_{\alpha_0}$ with $E_0=E(u_0)<0$. Assume that $\alpha_0$ is small enough so that all the estimates obtained in the previous sections hold true. Then we can prove Theorem \ref{MT1} in the following steps.

\subsection{Monotony properties}
In this subsection, we will derive from the local virial estimate \eqref{422} some monotony properties of the geometrical parameters. Roughly speaking, we will show that the parameter $b$ is always positive for large time. Since the scaling parameter $\lambda$ satisfies roughly
$$\frac{\lambda_s}{\lambda}\sim -b,$$
the sign structure of $b$ implies that $\lambda$ is almost monotony in time. This fact removes the possibility that the $H^{\frac{\beta}{2}}$ norm of the original solution $u(t)$ oscillates in time, which forces the solution to blow up in finite time. 

More precisely, we have
\begin{proposition}
\label{P5}
Assume that $\alpha_0$ is small enough and the spectral property \eqref{SP} holds true. Then there exists a unique $s_0\in[0,+\infty)$ such that
\begin{enumerate}
\item Sign structure of $b$:
\begin{equation}
\label{51}
\forall s<s_0,\; b(s)<0;\quad b(s_0)=0;\quad\forall s>s_0,\; b(s)>0.
\end{equation}
\item Monotonicity of $\lambda$: for all $s_2>s_1\geq s_0$, we have
\begin{equation}\label{52}
\frac{1}{2}\int_{s_1}^{s_2}b(s)\,ds-\delta(\alpha_0)\leq -\log\bigg(\frac{\lambda(s_2)}{\lambda(s_1)}\bigg)\leq \frac{3}{2}\int_{s_1}^{s_2}b(s)\,ds+\delta(\alpha_0),
\end{equation}
and 
\begin{equation}
\label{53}
\lambda(s_1)>\frac{1}{2}\lambda(s_2).
\end{equation}
\end{enumerate}
\end{proposition}

\begin{proof}We proceed the proof in the following steps.
\subsubsection*{\bf Step 1: }Equation of the scaling parameters.

We claim that for all $s_2>s_1\geq0$, there holds
\begin{equation}
\label{515}
\bigg|\log\bigg(\frac{\lambda(s_2)}{\lambda(s_1)}\bigg)+\int_{s_1}^{s_2}b\bigg|\leq \delta(\alpha_0)+\frac{1}{2}\int_{s_1}^{s_2}|b|
\end{equation}
We first differentiate \eqref{35} to obtain
\begin{align}
\label{516}
&(\partial_s\varepsilon_1,\partial_b\Theta)-(\partial_s\varepsilon_2,\partial_b\Sigma)+b_s\big[(\varepsilon_1,\partial^2_b\Theta)-(\varepsilon_2,\partial^2_b\Sigma)\big]\nonumber\\
&=-v_s\Big[\big(\varepsilon_1,\partial_b(\partial_v\Theta)\big)-\big(\varepsilon_2,\partial_b(\partial_v\Sigma)\big)\Big].
\end{align}
Then by projecting \eqref{312} and \eqref{313} onto $-\partial_b\Theta$ and $\partial_b\Sigma$ using  \eqref{23}, \eqref{310}, \eqref{319} and \eqref{328},
we can easily obtain the following rough estimate:
\begin{align}
\label{518}
&\bigg(\frac{\lambda_s}{\lambda}+b\bigg)\big\{(\Lambda\Sigma,\partial_b\Theta)-(\Lambda\Theta,\partial_b\Sigma)\big\}\nonumber\\
&= O\bigg(\lambda^\beta|E_0|+v^2+\int\big||D|^{\frac{\beta}{2}}\varepsilon\big|^2+\int|\varepsilon|^2e^{-|y|}+\delta(\alpha_0)|b|\bigg).
\end{align}
Since $(\Lambda\Sigma,\partial_b\Theta)-(\Lambda\Theta,\partial_b\Sigma)=(\Lambda Q_\beta,S_1)+O(|b|+|v|)\sim1$, we have
\begin{equation}\label{520}
\bigg(\frac{\lambda_s}{\lambda}+b\bigg)=O\bigg(\lambda^\beta|E_0|+v^2+\int\big||D|^{\frac{\beta}{2}}\varepsilon\big|^2+\int|\varepsilon|^2e^{-|y|}+\delta(\alpha_0)|b|\bigg).
\end{equation}
Integrating \eqref{520} from $s_1$ to $s_2$, using \eqref{423}, we obtain \eqref{515} immediately.

\subsubsection*{\bf Step 2: }Proof of \eqref{51}. 

Suppose that \eqref{51} does note hold. Then there are two possibilities:
\begin{enumerate}
\item There exists $s_0\in[0,+\infty)$ such that $b_s(s_0)\leq0$ and $b(s_0)=0$
\item For all $s\in[0,+\infty)$, we have $b(s)<0$.
\end{enumerate}
If the first one holds true, then from \eqref{422}, we have $\lambda^\beta(s_0)|E_0|\leq0$. Since $E_0<0$, we have $\lambda(s_0)=0$, which leads to a contradiction. Hence, we have for all $s\in[0,+\infty)$, $b(s)<0$.

If $\int_{s_1}^{+\infty}b=-\infty$, then from \eqref{515}, we have $\lambda(s_2)\rightarrow+\infty$ as $s_2\rightarrow+\infty$, which contradicts with \eqref{310}, \eqref{318} and the energy condition $E_0<0$. Thus, we have
\begin{equation}
\label{522}
\int_{s_1}^{+\infty}|b|<+\infty.
\end{equation}
Using \eqref{515} again, we can show that there exists $0<\lambda_-<\lambda_+$ such that for all $s\geq s_0$, we have
\begin{equation}
\label{521}
\lambda_-<\lambda(s)<\lambda_+.
\end{equation}
From \eqref{319}, \eqref{423} and \eqref{522}, we have:
$$\int_{s_1}^{\infty}|bb_s|\lesssim\int_{s_1}^{+\infty}\bigg( |b|+\lambda^\beta|E_0|+v^2+\int\big||D|^{\frac{\beta}{2}}\varepsilon\big|^2+\int|\varepsilon|^2e^{-|y|}\bigg)\,ds<+\infty.$$
Using \eqref{522} again, we have $b(s)\rightarrow0$, as $s\rightarrow+\infty$, and $b_s(s_n)\rightarrow 0$, as $n\rightarrow+\infty$, for some sequence $\{s_n\}^\infty_{n=1}$ with $s_n\rightarrow+\infty$, as $n\rightarrow+\infty$.

Combining all the results above and \eqref{422}, we obtain that
$$\lambda^\beta(s_n)|E_0|\rightarrow 0,\; as \;n\rightarrow+\infty,$$
which contradicts with \eqref{521} and the energy condition $E_0<0$. Hence, we conclude the proof of \eqref{51}.

\subsubsection*{\bf Step 3:} Proof of \eqref{52} and \eqref{53}.

Since for all $s\geq s_0$, we have $b(s)>0$, we know from \eqref{515} that
$$\int_{s_1}^{s_2}|b|=\int_{s_1}^{s_2}b.$$
Then \eqref{52} follows from \eqref{515} immediately.

Now, using \eqref{51} and \eqref{52}, we have
$$\log\bigg(\frac{\lambda(s_2)}{\lambda(s_1)}\bigg)\leq \delta(\alpha_0)-\int_{s_1}^{s_2}b(s)\,ds\leq \delta(\alpha_0)\leq e^{\frac{1}{2}},$$
which implies \eqref{53}. 

Now we conclude the proof of Proposition \ref{P5}.
\end{proof}

\subsection{Finite time or infinite time blow-up}
This subsection is devoted to show that the solution $u(t)$ must blow up either in finite time or in infinite time. More precisely, we claim that
\begin{equation}
\label{54}
\lambda(s)\rightarrow 0,\;as\;s\rightarrow+\infty.
\end{equation}

Indeed, from \eqref{422}, we have for all $s>s_0$, $b_s\geq-Cb^{10}$. Since $b(s)>0$ for all $s>s_0$, we have for all $s>s_0$, $b_s\geq-Cb^{10}$. After integration, we can show that for some $\tilde{s}_1>s_0$, there holds
\begin{equation}
\label{55}
b(s)\geq \frac{C}{s^{\frac{1}{9}}}.
\end{equation}
Hence from \eqref{52}, we have for all $s>\tilde{s}_1$
\begin{equation}
\label{56}
-\log\bigg(\frac{\lambda(s)}{\lambda(\tilde{s}_1)}\bigg)+\delta(\alpha_0)\geq\frac{1}{2}\int_{s_1}^{s_2}b(s)\rightarrow+\infty,\quad as\quad s\rightarrow+\infty,
\end{equation}
which implies \eqref{54} immediately. 
\subsection{Finite time blow-up and upper bound on the blow-up rate}
In this subsection, we will finish the proof of Theorem \ref{MT1}.

Combining \eqref{52} and \eqref{55}, we have for all $s>\tilde{s}_1$ ,
\begin{equation}
\label{57}
\log\bigg(\frac{\lambda(s)}{\lambda(\tilde{s}_1)}\bigg)\leq-\frac{1}{2}\int_{\tilde{s}_1}^{s}b(s)\,ds+\delta(\alpha_0)\leq C(\tilde{s}_1^{\frac{8}{9}}-s^{\frac{8}{9}})+\delta(\alpha_0).
\end{equation}
Hence, there exists some $\tilde{s}_2\gg\tilde{s}_1$, such that for all $s>\tilde{s}_2$
\begin{equation}\label{58}
\log\big(\lambda(s)\big)\leq -Cs^{\frac{8}{9}}\leq -Cb^{-8},
\end{equation}
for some universal constant $C>0$.

From \eqref{54}, we may chose a sequence of time $\{t_n\}$ such that $\lambda(t_n)=2^{-n}$ and $t_n\rightarrow T$, the maximal lifespan $T$, as $n\rightarrow+\infty$. Denote by $\tilde{t}_2$ such that $s(\tilde{t}_2)=\tilde{s}_2$ and $s_n$ such that $s(t_n)=s_n$. We also assume that for all $n\geq n_0$, there holds $t_n\geq \tilde{t}_2$. Recall from \eqref{53}, for all $n\geq n_0$ and $s\in[s_n,s_{n+1}]$, we have 
\begin{equation}
\label{510}
2^{-n-2}\leq\lambda(s)\leq 2^{-n+1}.
\end{equation}
From \eqref{52} and \eqref{58}, we have
\begin{align*}
\delta(\alpha_0)-\log\bigg(\frac{\lambda(s_{n+1})}{\lambda(s_n)}\bigg)\geq \int_{s_n}^{s_{n+1}}b(s)\,ds\geq C\int_{s_n}^{s_{n+1}}\frac{ds}{|\log(\lambda(s))|^{\frac{1}{8}}}
\end{align*}
Combining the above two estimates, we have
\begin{equation}
\label{59}
\int_{t_n}^{t_{n+1}}\frac{dt}{\lambda^\beta(t)|\log(\lambda(t))|^{\frac{1}{8}}}\leq C.
\end{equation}
it is easy to see from \eqref{510} that for all $n\geq n_0$, $t_n\leq t \leq t_{n+1}$ there holds
\begin{align}\label{511}
\frac{1}{C}2^{-\beta n}|\log(2^{-n})|^{\frac{1}{8}}\leq\lambda^\beta(t)|\log(\lambda(t))|^{\frac{1}{8}}\leq C 2^{-\beta n}|\log(2^{-n})|^{\frac{1}{8}},
\end{align}
for some universal constant $C>0$.

Injecting \eqref{511} into \eqref{59}, we have for all $n\geq n_0$,
\begin{equation}
\label{512}
t_{n+1}-t_n\leq C2^{-\beta n}|\log(2^{-n})|^{\frac{1}{8}}
\end{equation}
Summing \eqref{512} with respect to $n$ for all $n\geq n_0$, we obtain that $T<+\infty$ or equivalently, the solution blows up in finite time. On the other hand, we also obtain that for all $n\geq n_0$
\begin{equation}
\label{513}
T-t_n\leq C\sum_{k=n}^{+\infty}2^{-\beta k}k^{\frac{1}{8}}\leq C 2^{-\beta n}n^{\frac{1}{8}}\leq C\lambda^\beta(t_n)|\log(\lambda(t_n))|^{\frac{1}{8}},
\end{equation}
where we use the fact that $\lambda(t_n)=2^{-n}$ for the last inequality.

Next, for all $t$ close to $T$, there exists $n\geq n_0$ such that $t\in[t_n,t_{n+1}]$. From \eqref{510} and \eqref{513}, we have
\begin{equation}
\lambda^\beta(t)|\log(\lambda(t))|^{\frac{1}{8}}\geq C\lambda^\beta(t_n)|\log(\lambda(t_n))|^{\frac{1}{8}}\geq C(T-t_n)\geq C(T-t).
\end{equation}
Let $f(x)=x^\beta|\log x|^{\frac{1}{8}}$ for $x\in(0,x_0)$. It is easy to verify that for $x_0$ small enough, $f(x)$ is increasing in $x$. Since, for all $t$ close to $T$, there holds
\begin{align}\label{514}
f\Bigg(\bigg(\frac{T-t}{|\log(T-t)|^{\frac{1}{8}}}\bigg)^{\frac{1}{\beta}}\Bigg)&\leq \frac{T-t}{|\log(T-t)|^{\frac{1}{8}}}\bigg|\log\bigg(\frac{T-t}{|\log(T-t)|^{\frac{1}{8}}}\bigg)\bigg|^{\frac{1}{8}}\nonumber\\
&\leq C(T-t)\leq C\lambda^\beta(t)|\log(\lambda(t))|^{\frac{1}{8}}=Cf(\lambda(t)).
\end{align}
For $t$ close to $T$, we have $|\log(\lambda(t))|\gg1$, which implies $Cf\big(\lambda(t)\big)\leq f\big((2C)^{\frac{1}{\beta}}\lambda(t)\big)$.
Injecting this inequality into \eqref{514}, we have 
\begin{equation}
\lambda(t)\geq C\bigg(\frac{T-t}{|\log(T-t)|^{\frac{1}{8}}}\bigg)^{\frac{1}{\beta}},
\end{equation}
for all $t$ close to $T$, which  concludes the proof of Theorem \ref{MT1}.

\appendix

\section{Proof of the Hardy's type estimate}\label{A4}
This section is devoted to prove the Hardy's type estimate \eqref{327}. We first choose a cut-off function $\chi$ such that $\chi(y)=1$, if $|y|<1$; $\chi(y)=0$, if $|y|>2$. Then we have
\begin{equation}
\label{D1}
\int\frac{|f(y)|^2}{(1+|y|)^\beta}\,dy\lesssim\int\frac{|f(y)\chi(y)|^2}{(1+|y|)^\beta}\,dy+\int\frac{|f(y)(1-\chi(y))|^2}{(1+|y|)^\beta}\,dy:=\mathrm{I}+\mathrm{II}.
\end{equation}
For $I$ it is easy to obtain:
\begin{equation}
\label{D2}
\mathrm{I}\lesssim\int_{|y|\leq 2}|f|^2\lesssim\int|f|^2e^{-|y|}.
\end{equation}
While for $\mathrm{II}$, we first introduce the fractional Hardy's inequality in dimension $1$ introduced in \cite{BD}: for all $\alpha\in(0,2)$, we have
\begin{equation}
\label{D3}
\forall u\in C_c(D),\quad\frac{1}{2}\int_{D}\int_{D}\frac{|u(x)-u(y)|^2}{|x-y|^{d+\alpha}}dxdy\geq\kappa_{d,\alpha}\int_{D}\frac{|u(x)|^2}{x_d^\alpha}dx,
\end{equation}
where $D=\{(x_1,\ldots,x_d)|x_d>0\}$ is the upper half space and $\kappa_{d,\alpha}>0$ is given by%
\footnote{Here we use the following estimate: $\mathrm{B}(x,y)\sim \frac{1}{y}$, as $y\rightarrow 0^+.$}
\begin{equation}
\label{D16}
\kappa_{d,\alpha}=\frac{\pi^{\frac{d-1}{2}}\Gamma(\frac{1+\alpha}{2})[\mathrm{B}(\frac{1+\alpha}{2},\frac{2-\alpha}{2})-2^\alpha]}{\alpha 2^\alpha\Gamma(\frac{d+\alpha}{2})}\sim \frac{1}{2-\alpha},\quad as \alpha\rightarrow 2^-.
\end{equation}

On the other hand, from \cite[Proposition 4.1]{NPV}, we have the following characterization of the fractional Sobolev norm:
\begin{equation}
\label{D4}
\int\big||D|^{\frac{\beta}{2}}f\big|^2=C_\beta\int\int\frac{|f(x)-f(y)|^2}{|x-y|^{1+\beta}}\,dxdy,
\end{equation}
with $C_\beta\sim(2-\beta)$ as $\beta\rightarrow 2^-$.

Applying \eqref{D3} for $\alpha=\beta$, $d=1$, using \eqref{D16} and \eqref{D4}, we obtain for all $f\in H^{\frac{\beta}{2}}\cap C_c(\mathbb{R}\backslash \{0\})$.
\begin{align}
\label{D6}
\mathrm{II}&\leq \frac{1}{\kappa_{1,\beta}} \iint\frac{\big|f(x)(1-\chi(x))-f(y)(1-\chi(y))\big|^2}{|x-y|^{1+\beta}}\,dxdy\nonumber\\
&\leq \frac{2}{\kappa_{1,\beta}} \bigg(\iint\frac{|f(x)-f(y)|^2}{|x-y|^{1+\beta}}\,dxdy+  \iint\frac{|f(y)|^2|\chi(x)-\chi(y)|^2}{|x-y|^{1+\beta}}\,dxdy\bigg)\nonumber\\
&\leq C\int\big||D|^{\frac{\beta}{2}}f\big|^2+ \frac{1}{\kappa_{1,\beta}} \int|f(y)|^2\,dy\int\frac{|\chi(x)-\chi(y)|^2}{|x-y|^{1+\beta}}\,dx,
\end{align}
where $C$ is a universal constant independent of $\beta$.

We then denote by
\begin{align}
N_<=\frac{1}{\kappa_{1,\beta}}\int|f(y)|^2\,dy\int_{|t|\leq 1}\frac{|\chi(y)-\chi(y-t)|^2}{|t|^{1+\beta}}\,dt,\\
N_>=\frac{1}{\kappa_{1,\beta}}\int|f(y)|^2\,dy\int_{|t|\geq 1}\frac{|\chi(y)-\chi(y-t)|^2}{|t|^{1+\beta}}\,dt.
\end{align}
For $N_<$, using the Leibniz's rule we have
\begin{equation}
\label{D7}
N_<\leq\frac{1}{\kappa_{1,\beta}}\int|f(y)|^2\,dy\int_{|t|\leq 1}|t|^{-1-\beta}\bigg|t\int_0^1\chi'(y-t+st)\,ds\bigg|^2\,dt.
\end{equation}
Since, $\chi'(y)=0$, if $|y|>2$, we have for all $|t|\leq1$, $s\in[0,1]$ and $|y|\geq3$, $\chi'(y-t+st)=0,$
which implies that for all $|t|\leq1$, $s\in[0,1]$ and $y\in\mathbb{R}$, $|\chi'(y-t+st)|\lesssim e^{-|y|/2}$.
Combining with \eqref{D16} and\eqref{D7}, we obtain%
\footnote{Here, we use the fact that $\beta<2$.}
\begin{align}
\label{D9}
N_<\lesssim \bigg(\int|f(y)|^2e^{-|y|}\bigg)\bigg(\int_{|t|\leq 1}(2-\beta)|t|^{1-\beta}\,dt\bigg)\leq C\int|f(y)|^2e^{-|y|},
\end{align}
for some universal constant $C$ independent of $\beta$.

For $N_>$, we have
\begin{align}
\label{D10}
N_>\lesssim&\int|f(y)|^2\,dy\Bigg(\int_{|t|\geq 1}\frac{|\chi(y)|^2}{|t|^{1+\beta}}\,dt+\int_{|t|\geq 1}\frac{|\chi(y-t)|^2}{|t|^{1+\beta}}\,dt\Bigg)\nonumber\\
\lesssim&\int|f(y)\chi(y)|^2+\int_{|y|\leq A}|f(y)|^2\,dy\int_{|t|\geq 1}\frac{|\chi(y-t)|^2}{|t|^{1+\beta}}\,dt\nonumber\\
&+\int_{|y|\geq A}|f(y)|^2\,dy\int_{|t|\geq 1}\frac{|\chi(y-t)|^2}{|t|^{1+\beta}}\,dt\nonumber\\
:=&\widehat{\mathrm{I}}+\widehat{\mathrm{II}}+\widehat{\mathrm{III}},
\end{align}
with some large constant $A>10$ to be chosen later.

It is easy to obtain that
\begin{equation}
\widehat{\mathrm{I}}\lesssim \int|f(y)|^2e^{-|y|},\quad \widehat{\mathrm{II}}\lesssim \int_{|y|\leq A}|f(y)|^2\leq e^A\int|f(y)|^2e^{-|y|}.\label{D11}.
\end{equation}
While for $\widehat{\mathrm{III}}$, we have
\begin{align}
\label{D13}
\widehat{\mathrm{III}}\lesssim\int_{|y|\geq A}|f(y)|^2\,dy\int_{\{t||t|\geq 1,\,|y-t|\leq 2\}}|t|^{-1-\beta}\,dt.
\end{align}
Since $|y|\geq A>10$, we have for all $t$ with $|t|\geq 1$ and $|y-t|\leq 2$, there holds $|t|\geq \frac{1}{2}|y|$ and 
$$\int_{\{t||t|\geq 1,\,|y-t|\leq 2\}}|t|^{-1-\beta}\,dt\sim \frac{1}{|y|^{1+\beta}}$$
Thus, we have
\begin{equation}
\label{D14}
\widehat{\mathrm{III}}\lesssim \int_{|y|\geq A}\frac{|f(y)|^2}{|y|^{1+\beta}}\,dy\leq\frac{10}{A}\int\frac{|f(y)|^2}{(1+|y|)^\beta}\,dy.
\end{equation}
Combing \eqref{D11} and \eqref{D14}, we have
\begin{equation}
\label{D15}
N_>\leq C_A\int|f(y)|^2e^{-|y|}\,dy+\frac{10}{A}\int\frac{|f(y)|^2}{(1+|y|)^\beta}\,dy,
\end{equation}
with some constant $C_A>0$ depending only on $A$.

Finally, by collecting all the estimates above, we have
\begin{equation*}
\int\frac{|f(y)|^2}{(1+|y|)^\beta}\,dy\leq C_A\Bigg(\int\big||D|^{\frac{\beta}{2}}f\big|^2+\int|f(y)|^2e^{-|y|}\,dy\Bigg)+\frac{10}{A}\int\frac{|f(y)|^2}{(1+|y|)^\beta}\,dy.
\end{equation*}
Choosing $A$ large enough, we obtain \eqref{327} immediately.

\section{Proof of the spectral property}\label{A2}

This section is devoted to show that the quadratic form $H^\beta(\varepsilon,\varepsilon)$ is positive except on a dimension four subspace provided that $\beta$ is close to $2$. 

\subsection{Notations}
We start with some basic notations. For all $\varepsilon=\varepsilon_1+i\varepsilon_2\in H^{\frac{\beta}{2}}$, we denote by
\begin{align*}
&\overline{H}_1^\beta(\varepsilon_1,\varepsilon_1)=\int\big||D|^{\frac{\beta}{2}}\varepsilon_1\big|^2+\frac{10}{9}\bigg(2(2\beta+1)\int y\nabla Q_\beta Q_\beta^{2\beta-1}\varepsilon_1^2-\frac{1}{10}\int\frac{\varepsilon^2_1}{\cosh^2(\frac{10}{9}y)}\bigg),\\
&\overline{H}_2^\beta(\varepsilon_2,\varepsilon_2)=\int\big||D|^{\frac{\beta}{2}}\varepsilon_2\big|^2+\frac{10}{9}\bigg(2\int y\nabla Q_\beta Q_\beta^{2\beta-1}\varepsilon_2^2-\frac{1}{10}\int\frac{\varepsilon^2_2}{\cosh^2(\frac{10}{9}y)}\bigg),
\end{align*}
and 
\begin{align*}
&\overline{\mathcal{L}}^\beta_1=|D|^\beta+\frac{20(2\beta+1)}{9}y\nabla Q_\beta Q_\beta^{2\beta-1}-\frac{1}{9\cosh^2(\frac{10}{9}y)}\\
&\overline{\mathcal{L}}^\beta_2=|D|^\beta+\frac{20}{9}y\nabla Q_\beta Q_\beta^{2\beta-1}-\frac{1}{9\cosh^2(\frac{10}{9}y)}.
\end{align*}
Hence, we have
$$\overline{H}_1^\beta(\varepsilon_1,\varepsilon_1)=(\overline{\mathcal{L}}^\beta_1\varepsilon_1,\varepsilon_1),\quad \overline{H}_2^\beta(\varepsilon_2,\varepsilon_2)=(\overline{\mathcal{L}}^\beta_2\varepsilon_2,\varepsilon_2)$$
and
\begin{align}
\label{459}
H^\beta(\varepsilon,\varepsilon)=&\frac{1}{10}\bigg(\int\big||D|^{\frac{\beta}{2}}\varepsilon\big|^2+\int\frac{|\varepsilon|^2}{\cosh^2(\frac{10}{9}y)}\bigg)\nonumber\\
&\quad+\frac{9}{10}\big(\overline{H}_1^\beta(\varepsilon_1,\varepsilon_1)+\overline{H}_2^\beta(\varepsilon_2,\varepsilon_2)\big)
\end{align}
for all $\varepsilon=\varepsilon_1+i\varepsilon_2\in H^{\frac{\beta}{2}}$.

For simplicity, we denote by 
$$H(\varepsilon,\varepsilon)=(\mathcal{L}_1\varepsilon_1,\varepsilon_1)+(\mathcal{L}_2\varepsilon_2,\varepsilon_2),\quad\forall\varepsilon=\varepsilon_1+i\varepsilon_2\in H^1$$
with $\mathcal{L}_1=-\Delta+10yQ'Q^3$, $\mathcal{L}_2=-\Delta+2yQ'Q^{2\beta-1}$. 

We also denote by
\begin{align*}
&\overline{H}_1(\varepsilon_1,\varepsilon_1)=\int|\nabla\varepsilon_1|^2+\frac{10}{9}\bigg(10\int y\nabla Q Q^{3}\varepsilon_1^2-\frac{1}{10}\int\frac{\varepsilon^2_1}{\cosh^2(\frac{10}{9}y)}\bigg)\\
&\overline{H}_2(\varepsilon_2,\varepsilon_2)=\int|\nabla\varepsilon_2|^2+\frac{10}{9}\bigg(2\int y\nabla Q Q^{3}\varepsilon_2^2-\frac{1}{10}\int\frac{\varepsilon^2_2}{\cosh^2(\frac{10}{9}y)}\bigg),
\end{align*}
and
\begin{align*}
&\overline{\mathcal{L}}_1=-\Delta+\frac{100}{9}y\nabla Q Q^{3}-\frac{1}{9\cosh^2(\frac{10}{9}y)}\\
&\overline{\mathcal{L}}_2=-\Delta+\frac{20}{9}y\nabla Q Q^{3}-\frac{1}{9\cosh^2(\frac{10}{9}y)}
\end{align*}
for all $\varepsilon=\varepsilon_1+i\varepsilon_2\in H^1(\mathbb{R})$. 

Finally, we introduce the index of a bilinear form on a vector space $V$:
\begin{align*}
{\rm ind}_V(B):=\max\{k\in\mathbb{N}|&\text{there exists a subspace $P$ of codimension $k$} \\
& \text{such that $B_{|P}$ is positive.}\}.
\end{align*}
Let $H^1_e$ (respectively $H^1_o$) be the subspace of all even (respectively odd) $H^1$ functions. Assume that $H^1_e$ is $B$-orthogonal to $H^1_o$. We say that $B$ defined on $H^1$ has index $k+j$, if ${\rm ind}_{H^1_e}(B)=k$ and ${\rm ind}_{H^1_o}(B)=j$.

\subsection{Proof of the spectral property for $\beta$ close to $2$}
First, we claim that the following almost coercivity holds true for $\overline{H}_1^\beta$ and $\overline{H}_2^\beta$:
\begin{lemma}\label{L6}
For all $\kappa>0$, there exists $\beta_\kappa<2$ such that if $\beta_\kappa<\beta<2$, then for all $\varepsilon=\varepsilon_1+i\varepsilon_2\in H^{\frac{\beta}{2}}$ with $(\varepsilon_1,G_1)=(\varepsilon_1,Q_\beta)=0$ and $(\varepsilon_2,\Lambda Q_\beta+\frac{1}{2}\Lambda^2Q_\beta)=0$, there holds
\begin{equation}
\label{458}
\overline{H}_1^\beta(\varepsilon_1,\varepsilon_1)+\overline{H}_2^\beta(\varepsilon_2,\varepsilon_2)\geq -\kappa\bigg(\int\big||D|^{\frac{\beta}{2}}\varepsilon\big|^2+\int|\varepsilon|^2e^{-|y|}\bigg).
\end{equation}
\end{lemma}

We may easily see that \eqref{459} and \eqref{458} imply the spectral property \eqref{SP} immediately, when $\beta$ is close enough to $2$. On the other hand, by a standard density argument, we only need to show that \eqref{458} holds true for $\varepsilon=\varepsilon_1+i\varepsilon_2\in H^1(\mathbb{R})$. 

\begin{proof}[Proof of Lemma \ref{L6}]
First, we claim that $\overline{H}_1^\beta$ has index $1+1$ and $\overline{H}_2^\beta$ has index $1+0$, if $\beta$ is close enough to $2$.

From \cite[Lemma 10]{MR1}, we know that $\overline{H}_1$ has index $1+1$, while $\overline{H}_2$ has index $1+0$. Since we have $Q_\beta\rightarrow Q$ in $H^1$ as $\beta\rightarrow 2^-$. For $\beta$ close enough to $2$, we know that the number of the eigenvalues for $\overline{\mathcal{L}}_1^\beta$ ($\overline{\mathcal{L}}_2^\beta$ respectively) is the same as $\overline{\mathcal{L}}_1$ ($\overline{\mathcal{L}}_2$ respectively). Thus $\overline{H}_1^\beta$ must have index $1+1$ while $\overline{H}_2^\beta$ has index $1+0$, provided that $\beta$ is close enough to $2$.

Next, we state some numerical results obtained in \cite[Lemma 11]{MR1}.

\begin{lemma}[Numerical estimates]
For the operator $\overline{\mathcal{L}}_1$ and $\overline{\mathcal{L}}_2$, we have
\begin{enumerate}
\item There exists a unique even function $\phi_1\in L^\infty\cap \dot{H}^1$ such that $\overline{\mathcal{L}}_1\phi_1=Q$ and
\begin{equation}
\label{454}
-(\phi_1,Q)\bigg(1-\overline{H}_1(Q,Q)\frac{(\phi_1,Q)}{(Q,Q)^2}\bigg)>0.
\end{equation}
\item There exists a unique odd function $\phi_2\in L^\infty\cap \dot{H}^1$ such that $\overline{\mathcal{L}}_1\phi_2=yQ$ and
\begin{equation}
\label{455}
-(\phi_2,yQ)\bigg(1-\overline{H}_1(Q_y,Q_y)\frac{(\phi_2,yQ)}{(Q_y,yQ)^2}\bigg)>0.
\end{equation}
\item Let $\widetilde{Q}=\Lambda Q+\frac{1}{2}\Lambda^2Q$, then there exists a unique odd function $\phi_3\in L^\infty\cap \dot{H}^1$ such that $\overline{\mathcal{L}}_2\phi_3=\widetilde{Q}$ and
\begin{equation}
\label{456}
-(\phi_3,\widetilde{Q})\bigg(1-\overline{H}_2(Q,Q)\frac{(\phi_3,\widetilde{Q})}{(Q,\widetilde{Q})^2}\bigg)>0.
\end{equation}
\end{enumerate}
\end{lemma}
\begin{remark}
The estimates \eqref{454}--\eqref{456} are verified by numerical methods.
\end{remark}
\begin{remark}
As mentioned in \cite{MR1}, these three functions $\phi_j$ are not given by the Lax-Milgram theory, hence may not be in $L^2$. 
\end{remark}

Now,we turn back to the proof of \eqref{452}. Let $\chi$ be a smooth cut-off function such that $\chi(y)=1$ if $|y|<1$, $\chi(y)=0$ if $|y|>2$. We denote by 
$$(\phi_1)_A(y)=\phi_1(y)\chi(y/A)$$
for some $A>1$ to be chosen later. One may easily show that 
$$\|\nabla(\phi_1)_A\|_{L^2}\leq C$$
for some constant $C$ independent of $A$. We also have 
$$\Delta\phi_1=\frac{100}{9}y\nabla Q Q^{3}\phi_1-\frac{1}{9\cosh^2(\frac{10}{9}y)}\phi_1-Q\in L^2.$$
which implies that $\|\Delta(\phi_1)_A\|_{L^2}\leq C$ for some constant $C$ independent of $A$. Hence we have for all $s\in[1,2]$, there holds
\begin{align}
&\big\||D|^s(\phi_1)_A\big\|_{L^2}\leq C,\;\text{for some constant $C$ independent of $A$,}\label{469}\\
&\big\||D|^s\big((\phi_1)_A-\phi_1\big)\big\|_{L^2}\rightarrow 0,\;as\;A\rightarrow+\infty, \;\text{uniformly in $s$}.\label{470}
\end{align}

We first claim that for all $\kappa>0$, there exists $\beta_\kappa<2$ such that if $\beta_\kappa<\beta<2$, then for all real valued function $\varepsilon_1\in H^1_e$ with $(\varepsilon_1,Q_\beta)=0$, we have
\begin{align}
\overline{H}^\beta_1(\varepsilon_1,\varepsilon_1)\geq -\kappa\bigg(\int\big||D|^{\frac{\beta}{2}}\varepsilon_1\big|^2+\int|\varepsilon_1|^2e^{-|y|}\bigg).\label{457}
\end{align}

The proof of \eqref{457} can be divided into the following steps:
\begin{enumerate}
\item We consider the subspace $(P_1)_A\subset H^1_e$ spanned by $Q_\beta$ and $(\phi_1)_A$. We show that for $A$ large enough, there exists a $\beta(A)<2$ such that if $\beta\in(\beta(A),2)$, then $\overline{H}^\beta_1$ restricted to $(P_1)_A$ is not degenerate.
\item Let $(P_1)_A^\perp$ be the orthogonal of $(P_1)_A$ in $H^1_e$ with respect to the quadratic form $\overline{H}^\beta_1$. Then by an index argument, we can show that $\overline{H}^\beta_1$ is nonnegative on $(P_1)_A^\perp$, under the same assumption of the previous step.
\item For all $\varepsilon_1\in (P_1)_A$ with $(\varepsilon_1,Q_\beta)=0$, we show that $\overline{H}_1^\beta(\varepsilon_1,\varepsilon_1)>0$ provided that $A\geq A_0$ large enough and $\beta>\beta(A)$.
\item For all $(\varepsilon_1,Q_\beta)=0$, from (2) we have 
$$\overline{H}^\beta_1(\varepsilon_1,\varepsilon_1)=\overline{H}^\beta_1(\varepsilon^{(1)}_A,\varepsilon^{(1)}_A)+\overline{H}^\beta_1(\varepsilon^{(2)}_A,\varepsilon^{(2)}_A)\geq \overline{H}^\beta_1(\varepsilon^{(1)}_A,\varepsilon^{(1)}_A),$$
where $\varepsilon_1=\varepsilon^{(1)}_A+\varepsilon_A^{(2)}$ with $\varepsilon^{(1)}_A\in(P_1)_A$ and $\varepsilon_A^{(2)}\in(P_1)_A^\perp$. 
\item From (3), we have roughly
$$\overline{H}^\beta_1(\varepsilon^{(1)}_A,\varepsilon^{(1)}_A)\geq -C(\varepsilon_A^{(1)},Q_\beta)^2.$$
Using the fact that $Q_\beta$ is continuous in $H^1$ with respect to $\beta$, we can show that for $A$ large enough and $\beta>\beta(A)$, there holds
$$|(\varepsilon_A^{(1)},Q_\beta)|\leq \delta_A \bigg(\int\big||D|^{\frac{\beta}{2}}\varepsilon_1\big|^2+\int\big|\varepsilon_1\big|^2e^{-|y|}\bigg)^{\frac{1}{2}},$$
with some constant $\delta_A\rightarrow0$ as $A\rightarrow+\infty$. This concludes the proof of \eqref{457}.
\end{enumerate}

Now, let $(P_1)_A={\rm span}\{Q_\beta,(\phi_1)_A\}$. We need to show that if $A$ is large enough and $2>\beta>\beta(A)$ for some constant $\beta(A)<2$, then we have
\begin{equation}
\label{460}
\det
\begin{bmatrix}
\overline{H}_1^\beta(Q_\beta,Q_\beta)&\overline{H}_1^\beta(Q_\beta,(\phi_1)_A)\\
\overline{H}_1^\beta(Q_\beta,(\phi_1)_A)&\overline{H}_1^\beta((\phi_1)_A,(\phi_1)_A)
\end{bmatrix}
\not=0.
\end{equation}
Indeed, from the fact that $Q_\beta$ is continuous in $H^1$ with respect to $\beta$, we have
\begin{equation}
\overline{H}_1^\beta(Q_\beta,Q_\beta)=\overline{H}_1(Q,Q)+\delta(|2-\beta|).\label{461}
%\overline{H}_1^\beta(Q_\beta,(\phi_1)_A)=(\overline{\mathcal{L}}^\beta_1Q_\beta,(\phi_1)_A)=(\overline{\mathcal{L}}_1Q,(\phi_1)_A)
\end{equation}
Since $(\phi_1)_A\in H^1$, we have $\overline{\mathcal{L}}^\beta_1(\phi_1)_A\rightarrow\overline{\mathcal{L}}_1(\phi_1)_A$ in $L^2$, as $\beta\rightarrow 2^-$. But this convergence may not be uniform with respect to $A$, since we do not know whether $\phi_1\in \dot{H}^{\frac{\beta}{2}}$ for $\beta<2$. However, we still have for all fixed $A$ and $\kappa_0>0$, there exits a $\beta(A,\kappa_0)<2$ such that if $2>\beta>\beta(A,\kappa_0)$, then
\begin{align}
\label{462}
&\overline{H}_1^\beta(Q_\beta,(\phi_1)_A)=\overline{H}_1(Q,(\phi_1)_A)+O(\kappa_0),\\
&\overline{H}_1^\beta((\phi_1)_A,(\phi_1)_A)=\overline{H}_1((\phi_1)_A,(\phi_1)_A)+O(\kappa_0).\label{464}
\end{align}
On the other hand, we have
\begin{align}
\label{463}
&\det
\begin{bmatrix}
\overline{H}_1(Q,Q)&\overline{H}_1(Q,(\phi_1)_A)\\
\overline{H}_1(Q,(\phi_1)_A)&\overline{H}_1((\phi_1)_A,(\phi_1)_A)
\end{bmatrix}
\nonumber\\
&=-(Q,Q)^2\bigg(1-\overline{H}_1(Q,Q)\frac{(\phi_1,Q)}{(Q,Q)^2}\bigg)+O\bigg(\frac{1}{A}\bigg)\not=0,
\end{align}
provided that $A\geq A_0$ is large enough. Combining \eqref{461}--\eqref{463}, we obtain \eqref{460}, which concludes the proof of (1).

The proof of (2)--(4) is almost the same as \cite[Lemma 13]{MR1}, we omit the details here.

Now let  $\varepsilon^{(1)}_A=\alpha_AQ_\beta+\varepsilon^{(3)}_A,$
such that $(\varepsilon^{(3)}_A,Q_\beta)=0$ and $$\alpha_A=-\frac{(\varepsilon^{(1)}_A,Q_{\beta})}{(Q_\beta,Q_\beta)}.$$
Using (4), Cauchy-Schwarz inequality and the Hardy's type estimate \eqref{327}, we have for all $\kappa>0$, there exists a $C_\kappa>0$ such that
\begin{align}
\label{465}
\overline{H}^\beta_1(\varepsilon^{(1)}_A,\varepsilon^{(1)}_A)&=\overline{H}^\beta_1(\alpha_AQ_\beta+\varepsilon^{(3)}_A,\alpha_AQ_\beta+\varepsilon^{(3)}_A)\nonumber\\
&\geq -C_\kappa\alpha_A^2-\frac{\kappa}{2}\bigg(\int\big||D|^{\frac{\beta}{2}}\varepsilon_1\big|^2+\int|\varepsilon_1|^2e^{-|y|}\bigg).
\end{align}

Our goal is to show that for all $\kappa>0$, there exist $A(\kappa)\gg1$ and $\beta_\kappa<2$  such that if $A\geq A(\kappa)$ and $\beta\in(\beta_\kappa,2)$, then we have
\begin{equation}
\label{467}
|\alpha_A|\leq \sqrt{\frac{\kappa}{2C_\kappa}}\bigg(\int\big||D|^{\frac{\beta}{2}}\varepsilon_1\big|^2+\int|\varepsilon_1|^2e^{-|y|}\bigg)^{\frac{1}{2}},
\end{equation}
which together with (4) and \eqref{465}, implies \eqref{457} immediately.

Indeed, since $(\varepsilon_1,Q_\beta)=0$, we have $\alpha_A=\frac{(\varepsilon^{(2)}_A,Q_{\beta})}{(Q_\beta,Q_\beta)}$. From the definition of $(P_1)^\perp_A$, we have:
\begin{align}
\label{466}
&|(\varepsilon^{(2)}_A,Q_{\beta})|=\big|\big(\varepsilon^{(2)}_A,Q_{\beta}-\overline{\mathcal{L}}^\beta_1(\phi_1)_A\big)\big|\nonumber\\
&\leq |(\varepsilon^{(2)}_A,Q_\beta-Q)|+\big|\big(\varepsilon^{(2)}_A,(\overline{\mathcal{L}}^\beta_1-\overline{\mathcal{L}}_1)(\phi_1)_A\big)\big|+\big|\big(\varepsilon^{(2)}_A,\overline{\mathcal{L}}_1((\phi_1)_A-\phi_1)\big)\big|,
\end{align}
where we use the fact that 
$$0=\overline{H}_1^\beta(\varepsilon^{(2)}_A,(\phi_1)_A)=(\varepsilon^{(2)}_A,\overline{\mathcal{L}}^\beta_1(\phi_1)_A)$$
and $\overline{\mathcal{L}}_1\phi_1=Q$ for \eqref{466}. From Lemma \ref{L7}, we have for $A\geq A_0$ large enough and $2>\beta>\beta(A)$, there exists a small constant $\delta_A$, with $\delta_A\rightarrow 0$ as $A\rightarrow+\infty$, such that
\begin{align}
\label{473}
|\alpha_A|&\leq \delta_A\bigg(\int\big||D|^{\frac{\beta}{2}}\varepsilon_A^{(2)}\big|^2+\int\big|\varepsilon_A^{(2)}\big|^2e^{-|y|}\bigg)^{\frac{1}{2}}\nonumber\\
&\leq C\delta_A\bigg(\int\big||D|^{\frac{\beta}{2}}\varepsilon_1\big|^2+\int\big|\varepsilon_1\big|^2e^{-|y|}\bigg)^{\frac{1}{2}},
\end{align}
which concludes the proof of \eqref{467}, hence the proof of \eqref{457}.

Similar as \eqref{457}, we have for all $\kappa>0$, there exists $\beta_\kappa<2$ such that if $\beta_\kappa<\beta<2$, then for all real valued function $\varepsilon_1\in H^1_o$ and $\varepsilon_2\in H^1_e$, with $(\varepsilon_1,G_1)=0$ and $(\varepsilon_2,\Lambda Q_\beta+\frac{1}{2}\Lambda^2 Q_\beta)=0$, there holds
\begin{align}
&\overline{H}^\beta_1(\varepsilon_1,\varepsilon_1)\geq -\kappa\bigg(\int\big||D|^{\frac{\beta}{2}}\varepsilon_1\big|^2+\int|\varepsilon_1|^2e^{-|y|}\bigg),\label{479}\\
&\overline{H}^\beta_2(\varepsilon_2,\varepsilon_2)\geq -\kappa\bigg(\int\big||D|^{\frac{\beta}{2}}\varepsilon_2\big|^2+\int|\varepsilon_2|^2e^{-|y|}\bigg),\label{480}
\end{align}
which concludes the proof of Proposition \ref{P6}.

\end{proof}

\bibliographystyle{amsplain}
\bibliography{ref}

\providecommand{\bysame}{\leavevmode\hbox to3em{\hrulefill}\thinspace}
\providecommand{\MR}{\relax\ifhmode\unskip\space\fi MR }
\providecommand{\MRhref}[2]{%
  \href{http://www.ams.org/mathscinet-getitem?mr=#1}{#2}
}
\providecommand{\href}[2]{#2}
\begin{thebibliography}{10}

\bibitem{AT}
C.~J. Amick and J.~F. Toland, \emph{Uniqueness and related analytic properties
  for the {B}enjamin-{O}no equation--a nonlinear {N}eumann problem in the
  plane}, Acta Math. \textbf{167} (1991), no.~1, 107--126.

\bibitem{BD}
K.~Bogdan and B.~Dyda, \emph{The best constant in a fractional {H}ardy
  inequality}, Math. Nachr. \textbf{5} (2011), no.~284, 629--638.

\bibitem{BHL}
T.~Boulenger, D.~Himmelsbach, and E.~Lenzmann, \emph{Blowup for fractional
  {N}{L}{S}}, J. Funct. Anal. \textbf{271} (2016), no.~9, 2569--2603.

\bibitem{BW}
J.~Bourgain and W.~Wang, \emph{Construction of blowup solutions for the
  nonlinear {S}chr{\"o}dinger equation with critical nonlinearity}, Ann. Sc.
  Norm. Super. Pisa Cl. Sci. \textbf{25} (1997), no.~1-2, 197--215.

\bibitem{CMMT}
D.~Cai, A.~J. Majda, D.~W. McLaughlin, and E.~G. Tabak, \emph{Dispersive wave
  turbulence in one dimension}, Phys. D \textbf{152} (2001), 551--572.

\bibitem{TW}
T.~Cazenave and F.~Weissler, \emph{Some remarks on the nonlinear
  {S}chr\"odinger equation in the critical case.}, vol. 1394, pp.~18--29,
  Springer-Verlag, 1989.

\bibitem{CHHO}
Y.~Cho, H.~Hajaiej, G.~Hwang, and T.~Ozawa, \emph{On the {C}auchy problem of
  fractional {S}chr{\"o}dinger equation with {H}artree type nonlinearity},
  Funkcial. Ekvac. \textbf{56} (2013), no.~2, 193--224.

\bibitem{CHKL}
Y.~Cho, G.~Hwang, S.~Kwon, and S.~Lee, \emph{Profile decompositions and blowup
  phenomena of mass critical fractional {S}chr{\"o}dinger equations}, Nonlinear
  Anal. \textbf{86} (2013), 12--29.

\bibitem{CP}
A.~Choffrut and O.~Pocovnicu, \emph{Ill-posedness of the cubic nonlinear
  half-wave equation and other fractional {N}{L}{S} on the real line}, Int.
  Math. Res. Not. IMRN \textbf{2018} (2016), no.~3, 699--738.

\bibitem{NPV}
E.~Di~Nezza, G.~Palatucci, and E.~Valdinoci, \emph{Hitchhikers guide to the
  fractional {S}obolev spaces}, Bull. Sci. Math. \textbf{136} (2012), no.~5,
  521--573.

\bibitem{DOD4}
B.~Dodson, \emph{Global well-posedness and scattering for the mass critical
  nonlinear {S}chr{\"o}dinger equation with mass below the mass of the ground
  state}, Adv. Math. \textbf{285} (2015), 1589--1618.

\bibitem{ES}
W.~Eckhaus and P.~Schuur, \emph{The emergence of solitons of the {K}orteweg-de
  {V}ries equation from arbitrary initial conditions}, Math. Methods Appl. Sci.
  \textbf{5} (1983), no.~1, 97--116.

\bibitem{FL}
R.~Frank and E.~Lenzmann, \emph{Uniqueness of non-linear ground states for
  fractional {L}aplacians in $\mathbb{R}$}, Acta Math. \textbf{210} (2013),
  no.~2, 261--318.

\bibitem{FLS}
R.~Frank, E.~Lenzmann, and L.~Silvestre, \emph{Uniqueness of radial solutions
  for the fractional {L}aplacian}, Comm. Pure and Appl. Math. \textbf{69}
  (2016), no.~9, 1671--1726.

\bibitem{GH}
B.~Guo and Z.~Huo, \emph{Global well-posedness for the fractional nonlinear
  {S}chr\"odinger equation}, Comm. Partial Differential Equations \textbf{36}
  (2011), no.~2, 247--255.

\bibitem{HS}
Y.~Hong and Y.~Sire, \emph{On fractional {S}chr\"odinger equations in {S}obolev
  spaces}, arXiv preprint arXiv:1501.01414 (2015).

\bibitem{IP}
A.~D. Ionescu and F.~Pusateri, \emph{Nonlinear fractional {S}chr{\"o}dinger
  equations in one dimension}, J. Funct. Anal. \textbf{266} (2014), no.~1,
  139--176.

\bibitem{KMR1}
C.E. Kenig, Y.~Martel, and L.~Robbiano, \emph{Local well-posedness and blow-up
  in the energy space for a class of ${L}^2$ critical dispersion generalized
  {B}enjamin–{O}no equations}, Ann. Inst. H. Poincar\'e Anal. Non Lin\'eaire
  \textbf{28} (2011), 853--887.

\bibitem{KLS}
K.~Kirkpatrick, E.~Lenzmann, and G.~Staffilani, \emph{On the continuum limit
  for discrete {N}{L}{S} with long-range lattice interactions}, Comm. Math.
  Phys. \textbf{317} (2013), no.~3, 563--591.

\bibitem{KSM}
C.~Klein, C.~Sparber, and P.~Markowich, \emph{Numerical study of fractional
  nonlinear {S}chr{\"o}dinger equations}, Proc. R. Soc. A \textbf{470} (2014),
  no.~2172, 20140364.

\bibitem{KLR}
J.~Krieger, E.~Lenzmann, and P.~Rapha{\"e}l, \emph{Nondispersive solutions to
  the ${L}^2$-critical half-wave equation}, Arch. Ration. Mech. Anal.
  \textbf{209} (2013), no.~1, 61--129.

\bibitem{KS}
J.~Krieger and W.~Schlag, \emph{Non-generic blow-up solutions for the critical
  focusing {N}{L}{S} in 1-{D}}, J. Eur. Math. Soc. (JEMS) \textbf{11} (2009),
  no.~1, 1--125.

\bibitem{LPSS}
M.~J. Landman, G.~C. Papanicolaou, C.~Sulem, and P.~L. Sulem, \emph{Rate of
  blowup for solutions of the nonlinear {S}chr{\"o}dinger equation at critical
  dimension}, Phys. Rev. A \textbf{38} (1988), no.~8, 3837.

\bibitem{MMT2}
A.~J. Majda, D.~W. McLaughlin, and E.~G. Tabak, \emph{A one-dimensional model
  for dispersive wave turbulence}, J. Nonlinear Sci. \textbf{7} (1997), no.~1,
  9--44.

\bibitem{MM1}
Y.~Martel and F.~Merle, \emph{A {L}iouville theorem for the critical
  generalized {K}orteweg-de {V}ries equation}, J. Math. Pures Appl. \textbf{79}
  (2000), no.~4, 339--425.

\bibitem{MM2}
\bysame, \emph{Instability of solitons for the critical generalized
  {K}orteweg-de {V}ries equation}, Geom. Funct. Anal. \textbf{11} (2001),
  no.~1, 74--123.

\bibitem{MM4}
\bysame, \emph{Blow up in finite time and dynamics of blow up solutions for the
  ${L}^2$--critical generalized {K}d{V} equation}, J. Amer. Math. Soc.
  \textbf{15} (2002), no.~3, 617--664.

\bibitem{MM5}
\bysame, \emph{Nonexistence of blow-up solution with minimal ${L}^2$-mass for
  the critical g{K}d{V} equation}, Duke Math. J. \textbf{115} (2002), no.~2,
  385--408.

\bibitem{MM3}
\bysame, \emph{Stability of blow-up profile and lower bounds for blow-up rate
  for the critical generalized {K}d{V} equation}, Ann. of Math. \textbf{155}
  (2002), no.~1, 235--280.

\bibitem{MMR1}
Y.~Martel, F.~Merle, and P.~Rapha{\"e}l, \emph{Blow up for the critical
  generalized {K}orteweg-de {V}ries equation {I}: {D}ynamics near the soliton},
  Acta Math. \textbf{212} (2014), no.~1, 59--140.

\bibitem{MMR2}
\bysame, \emph{Blow up for the critical g{K}d{V} equation {I}{I}: minimal mass
  blow up}, J. Eur. Math. Soc. (JEMS) \textbf{17} (2015), no.~8, 1855--1925.

\bibitem{MMR3}
\bysame, \emph{Blow up for the critical g{K}d{V} equation {I}{I}{I}: exotic
  regimes}, Ann. Sc. Norm. Super. Pisa Cl. Sci. \textbf{XIX} (2015), no.~2,
  575--631.

\bibitem{MP}
Y.~Martel and D.~Pilod, \emph{Construction of a minimal mass blow up solution
  of the modified {B}enjamin-{O}no equation}, Math. Ann. \textbf{369} (2017),
  no.~1-2, 153--245.

\bibitem{MarR}
Y.~Martel and P.~Rapha\"{e}l, \emph{Strongly interacting blow up bubbles for
  the mass critical {N}{L}{S}}, Ann. Sci. \'Ecole Norm. Sup. \textbf{51}
  (2018), no.~3, 701--737.

\bibitem{M7}
F.~Merle, \emph{Determination of blow-up solutions with minimal mass for
  nonlinear {S}chr{\"o}dinger equations with critical power}, Duke Math. J.
  \textbf{69} (1993), no.~2, 427--454.

\bibitem{M1}
\bysame, \emph{Existence of blow-up solutions in the energy space for the
  critical generalized {K}d{V} equation}, J. Amer. Math. Soc. \textbf{14}
  (2001), no.~3, 555--578.

\bibitem{MR3}
F.~Merle and P.~Rapha\"{e}l, \emph{Sharp upper bound on the blow-up rate for
  the critical nonlinear {S}chr{\"o}dinger equation}, Geom. Funct. Anal.
  \textbf{13} (2003), no.~3, 591--642.

\bibitem{MR2}
\bysame, \emph{On universality of blow-up profile for ${L}^2$ critical
  nonlinear {S}chr{\"o}dinger equation}, Invent. Math. \textbf{156} (2004),
  no.~3, 565--672.

\bibitem{MR1}
\bysame, \emph{The blow-up dynamic and upper bound on the blow-up rate for
  critical nonlinear {S}chr\"{o}dinger equation}, Ann. of Math. \textbf{161}
  (2005), no.~1, 157--220.

\bibitem{MR5}
\bysame, \emph{Profiles and quantization of the blow up mass for critical
  nonlinear {S}chr{\"o}dinger equation}, Comm. Math. Phys. \textbf{253} (2005),
  no.~3, 675--704.

\bibitem{MRS3}
F.~Merle, P.~Rapha{\"e}l, and J.~Szeftel, \emph{The instability of
  {B}ourgain-{W}ang solutions for the ${L}^2$ critical {N}{L}{S}}, Amer. J.
  Math. \textbf{135} (2013), no.~4, 967--1017.

\bibitem{NP}
I.~Naumkin and P.~Rapha{\"e}l, \emph{On small traveling waves to the mass
  critical fractional {N}{L}{S}}, Calc. Var. Partial Differential Equations
  \textbf{57} (2018), no.~3, 93.

\bibitem{Per}
G.~Perelman, \emph{On the formation of singularities in solutions of the
  critical nonlinear {S}chr{\"o}dinger equation}, Ann. Henri. Poincar{\'e},
  vol.~2, 2001, pp.~605--673.

\bibitem{MR4}
P.~Rapha{\"e}l, \emph{Stability of the $\log$-$\log$ bound for blow up
  solutions to the critical nonlinear {S}chr{\"o}dinger equation}, Math. Ann.
  \textbf{331} (2005), no.~3, 577--609.

\bibitem{SS}
C.~Sulem and P.~L. Sulem, \emph{The nonlinear {S}chr{\"o}dinger equation:
  self-focusing and wave collapse}, vol. 139, Springer Science \& Business
  Media, 2007.

\bibitem{W1}
M.~I. Weinstein, \emph{Nonlinear {S}chr{\"o}dinger equations and sharp
  interpolation estimates}, Comm. Math. Phys. \textbf{87} (1983), no.~4,
  567--576.

\bibitem{W4}
\bysame, \emph{Existence and dynamic stability of solitary wave solutions of
  equations arising in long wave propagation}, Comm. Partial Differential
  Equations \textbf{12} (1987), no.~10, 1133--1173.

\end{thebibliography}
\end{document}